\documentclass[12pt]{article}
\author{Yiqiao Zhong\\ORFE\\Princeton University \and Nicolas Boumal\\Mathematics Department and PACM\\Princeton University}
\usepackage{url}
\usepackage{graphicx}
\usepackage{tikz}
\usetikzlibrary{decorations.pathreplacing,calc}

\include{preamble_arxiv}

\newcommand{\Arg}{\mathrm{Arg}}
\newcommand{\1}{\mathbf{1}}

\newcommand{\calN}{\mathcal{N}}
\newcommand{\calT}{\mathcal{T}}
\newcommand{\etheta}{e^{i \theta}}
\newcommand{\mtheta}{\min_{\theta}}
\newcommand{\ntikzmark}[2]{#2\thinspace\tikz[overlay,remember picture,baseline=(#1.base)]{\node[inner sep=0pt] (#1) {};}}

\newcommand{\makebrace}[3]{%
    \begin{tikzpicture}[overlay, remember picture]
        \draw [decoration={brace,amplitude=0.5em},decorate]
        let \p1=(#1), \p2=(#2) in
        ({max(\x1,\x2)}, {\y1+0.8em}) -- node[right=0.6em] {#3} ({max(\x1,\x2)}, {\y2});
    \end{tikzpicture}
}

\renewcommand{\Re}{\operatorname{Re}}
\renewcommand{\Im}{\operatorname{Im}}

\title{Near-optimal bounds for phase synchronization}

\date{\today}

\begin{document}

\sloppy

\maketitle

\begin{abstract}
The problem of phase synchronization is to estimate the phases (angles) of a complex unit-modulus vector $z$ from their noisy pairwise relative measurements $C = zz^* + \sigma W$, where $W$ is a complex-valued Gaussian random matrix. The maximum likelihood estimator (MLE) is a solution to a unit-modulus constrained quadratic programming problem, which is nonconvex. Existing works have proposed polynomial-time algorithms such as a semidefinite relaxation (SDP) approach  or the generalized power method (GPM) to solve it. Numerical experiments suggest both of these methods succeed with high probability for $\sigma$ up to $\tilde{\mathcal{O}}(n^{1/2})$, yet, existing analyses only confirm this observation for $\sigma$ up to $\mathcal{O}(n^{1/4})$. In this paper, we bridge the gap, by proving SDP is tight for $\sigma = \mathcal{O}(\sqrt{n /\log n})$, and GPM converges to the global optimum under the same regime. Moreover, we establish a linear convergence rate for GPM, and derive a tighter $\ell_\infty$ bound for the MLE. A novel technique we develop in this paper is to track (theoretically) $n$ closely related sequences of iterates, in addition to the sequence of iterates GPM actually produces. As a by-product, we obtain an $\ell_\infty$ perturbation bound for leading eigenvectors. 
Our result also confirms intuitions that use techniques from statistical mechanics.
\end{abstract}



\section{Introduction}\label{sec:intro}

Phase synchronization is the problem of estimating $n$ angles $\theta_1, \ldots, \theta_n$ in $[0, 2\pi)$ based on noisy measurements of their differences $\theta_k - \theta_\ell \textrm{ mod } 2\pi$. This is equivalent to estimating $n$ phases $e^{i\theta_1}, \ldots, e^{i\theta_n}$ from measurements of relative phases $e^{i(\theta_k - \theta_\ell)}$.

A typical noise model for this estimation problem is as follows. The target parameter (the signal) is the vector $z\in\Cn$ with entries $z_k = e^{i\theta_k}$. The measurements are stored in a matrix $C \in \Cnn$ such that, for $k < \ell$,
\begin{align}
	C_{k\ell} & = z_k \bar z_\ell + \sigma W_{k\ell},
	\label{eq:Ckl}
\end{align}
where $\sigma \geq 0$ is the noise level and $\{W_{k\ell}\}_{k < \ell}$ are independent standard complex Gaussian variables. Under this model, defining $W_{kk} = 0$ and $W_{\ell k} = \overline{W_{k\ell}}$ for consistency, the model is compactly written in matrix notation as
\begin{align}
	C & = zz^* + \sigma W,
	\label{eq:C}
\end{align}
where both $C$ and $W$ are Hermitian. An easy derivation\footnote{Since $W$ is Gaussian, an MLE minimizes the squared Frobenius norm: $\|C-xx^*\|_F^2 = \|C\|_F^2 + \|xx^*\|_F^2 - 2 x^*Cx$. Owing to $|x_k| = 1 \forall k$, this is equivalent to maximizing $x^* Cx$.} shows that a maximum likelihood estimator (MLE) $\hat x \in \Cn$ for the signal $z$ is a global optimum of the following quadratically constrained quadratic program (we define $[n] = \{1, \ldots, n\}$):
\begin{align}
	\max_{x\in\Cn} \ \ x^* C x \ \textrm{ subject to } |x_k| = 1 \textrm{ for } k \in [n].
	\tag{P}
	\label{eq:P}
\end{align}
%
%
Problem~\eqref{eq:P} is non-convex and hard in general (\cite[Prop.~3.5]{zhang2006complex}). Yet, numerical experiments in~\citep{bandeira2014open,bandeira2014tightness} suggest that, provided $\sigma = \tilde{\mathcal{O}}(\sqrt{n})$,\footnote{The notation $\tilde{\mathcal{O}}$ suppresses potential log factors.} the following convex semidefinite relaxation for~\eqref{eq:P} admits $\hat x \hat x^*$ as its unique global optimum with high probability (more generally, if the problem below admits a solution of rank 1, $\hat x \hat x^*$, then the relaxation is said to be tight and $\hat x$ is an optimum of~\eqref{eq:P}):
\begin{align}
	\max_{X \in \Cnn, X = X^*} \ \trace(CX) \ \textrm{  subject to } \diag(X) = \1, X \succeq 0.
\tag{SDP}
\label{eq:SDP}
\end{align}
In this paper, we give a rigorous proof for this observation, improving on the previous best result which only handles $\sigma = \mathcal{O}(n^{1/4})$~\citep{bandeira2014tightness}. Our result also provides some justification for the analytical prediction in \cite{javanmard2016phase} on optimality of the semidefinite relaxation approach.\footnote{To be precise, in \cite{javanmard2016phase}, it is predicted---but not proved---via statistical mechanics arguments that the SDP relaxation is nearly optimal when $\sigma = \mathcal{O}(\sqrt{n})$. Instead of showing a solution of (\ref{eq:SDP}) has rank one, a rescaled leading eigenvector of its solution is used as an estimator.} 
\begin{theorem}\label{thm:mainSDP}
	If $\sigma = \mathcal{O}\left(\sqrt{\frac{n}{\log n}}\right)$, with high probability for large $n$, the semidefinite program~\eqref{eq:SDP} admits a unique solution $\hat x \hat x^*$, where $\hat x$ is a global optimum of~\eqref{eq:P} (unique up to phase.)
\end{theorem}
In the theorem statement, ``up to phase'' refers to the fact that the measurements $C$ are relative: the distribution of $C$ does not change if $z$ is replaced by $ze^{i\theta}$ for any angle $\theta$, so that if $\hat x$ is a solution of~\eqref{eq:P}, then necessarily so is $\hat x e^{i\theta}$ for any $\theta$, and $z$ can only be recovered up to a global phase shift.

Theorem~\ref{thm:mainSDP} shows that the non-convex problem~\eqref{eq:P} enjoys what is sometimes called \emph{hidden convexity}, that is, in the proper noise regime, it is equivalent to a (tractable) convex problem. 
As a consequence, it is not a hard problem in that regime, suggesting local solvers may be able to solve it in its natural dimension.  This is desirable, since the relaxation~\eqref{eq:SDP}, while convex, has the disadvantage of lifting the problem from $n-1$ to $n(n-1)$ dimensions.

And indeed, numerical experiments in~\citep{boumal2016nonconvexphase} suggest that local optimization algorithms applied to~\eqref{eq:P} directly succeed in the same regime as~\eqref{eq:SDP}. This was confirmed theoretically in~\citep{boumal2016nonconvexphase} for $\sigma = \mathcal{O}(n^{1/6})$, using both a modification of the power method called the generalized power method (GPM) and local optimization algorithms acting directly on the search space of~\eqref{eq:P}, which is a manifold. Results pertaining to GPM have been rapidly improved to allow for $\sigma = \mathcal{O}(n^{1/4})$ in~\citep{liu2016statistical}.

In this paper, we consider a version of GPM listed as Algorithm~\ref{algo:GPM} and prove that it works in the same regime as the semidefinite relaxation, thus better capturing the empirical observation. Note that GPM, as a local algorithm, is a more desirable approach versus semidefinite relaxation in practice. GPM and its variants are also considered in a number of related problems \citep{journee2010generalized, deshpande2014conePCA, roulet2015renegar, chen2016projected}, and can be seen as special cases of the conditional gradient algorithm \citep{journee2010generalized, luss2013conditional}.

\begin{theorem}\label{thm:mainGPM}
	If $\sigma = \mathcal{O}\left(\sqrt{\frac{n}{\log n}}\right)$, with high probability for large $n$, Algorithm~\ref{algo:GPM} converges at least linearly to the global optimum of~\eqref{eq:P} (unique up to phase.)
\end{theorem}

\begin{algorithm}[t]
	\caption{Generalized power method (GPM) without shift}\label{algo:GPM}
	\begin{algorithmic}[1]
		\State \textbf{Input:} Hermitian measurement matrix $C \in \Cnn$. 
		\State \textbf{Initialize:} Set $x^0$ to be a leading eigenvector of $C$ with $\|x^0\|_2 = \sqrt{n}$.
		\For{$t = 1,2\ldots$} 
		\State $x^t = \mathcal{P}(Cx^{t-1})$ \Comment{Apply entrywise $\mathcal{P}(a) = \begin{cases}
			a/|a| & \textrm{if } a\neq 0, \\ 1 & \textrm{otherwise.}
			\end{cases}$}
		\EndFor
	\end{algorithmic}
\end{algorithm}



To establish both results, we develop an original proof technique based on following $n+1$ separate but closely related sequences of feasible points for~\eqref{eq:P}, designed so that they will have suitable statistical independence properties. Furthermore, as a necessary step toward proving the main theorems, we prove an $\ell_\infty$ perturbation bound for eigenvectors, which is of independent interest.

It is worth noting that for $\sigma > \sqrt{n}$, it is impossible to reliably detect, with probability tending to 1, whether $C$ is of the form $zz^* + \sigma W$ or if it is only of the form $\sigma W$~\citep[Thm.~6.11]{perry2016optimality}, which suggests that $\sigma < \sqrt{n}$ is necessary in order for a good estimator to exist. This can be made precise by considering the simpler problem of $\mathbb{Z}_2$ synchronization,\footnote{The problem is formulated as follows: $z \in \{-1, 1\}^n$, noise $W$ is a real random matrix, e.g., a Gaussian Wigner matrix, and the goal is to recover $z$ from $C = zz^T + \sigma W$. } where we have the stronger knowledge that $z_k \in \{\pm 1\}$. 
For the $\mathbb{Z}_2$ synchronization problem, non-rigorous arguments that use techniques from statistical mechanics show $\sigma = \sqrt{n}$ is the information-theoretic threshold for mean squared estimation error (MSE):  when $\sigma$ is above this threshold, no estimator is able to beat the trivial estimator $x = 0$ as $n \to \infty$ \citep{javanmard2016phase}. In \cite{lelarge2016fundamental}, it was rigorously proved that $\sigma = \sqrt{n}$ is the threshold for a different notion of MSE. These results a fortiori imply, for phase synchronization, that $\sigma = \mathcal{O}(\sqrt{n})$ is necessary\footnote{Suppose the prior is supported and uniformly distributed on $\{-1, 1\}^n$. By independence, the Bayes-optimal estimator for phase synchronization is a function of $\Re(C)$, so we can use information-theoretic results about the $\mathbb{Z}_2$ synchronization problem.} in order for an estimator to have nontrivial MSE (better than the trivial estimator $x=0$). It is also known that both the eigenvector estimator and the MLE have nontrivial MSE as soon as $\sigma < \sqrt{n}$ \citep{capitaine2009largest, Ben11, javanmard2016phase}. Whether the extra logarithmic factor is necessary to compute the MLE efficiently up to the threshold remains to be determined.

To close this introduction, we state the relevance of the MLE $\hat x$ as an estimator for $z$.
\begin{theorem}\label{thm:mainbound}
 Let $\hat x$ be a global optimum of~\eqref{eq:P}, with global phase such that $z^*\hat x = |z^*\hat x|$.  Then, deterministically,
 \begin{align}
	 \|\hat x - z\|_2 & \leq 4\sigma \frac{\|W\|_2}{\sqrt{n}}.
 \end{align}
 Furthermore, if $\sigma = \mathcal{O}(\sqrt{n/\log n})$, then with high probability for large $n$, 
 \begin{align}
	 \|\hat x - z\|_2 & = \mathcal{O}(\sigma), \textrm{ and} \\
	 \|\hat x - z\|_\infty & = \mathcal{O}(\sigma \sqrt{\log n/n}).
 \end{align}
\end{theorem}
The bound on $\ell_2$ error appears in~\citep[Lem.~4.1]{bandeira2014tightness}, while the bound on $\ell_\infty$ error improves on~\citep[Lem.~4.2]{bandeira2014tightness} as a by-product of the results obtained here. We remark that $\sigma = \mathcal{O}(\sqrt{n/\log n})$ is necessary for a nontrivial $\ell_\infty$ error (smaller than $1$, which is trivially attained by $x = 0$) due to \cite{bandeira2014tightness, bandeira2015laplacian}.\footnote{In \cite{bandeira2014tightness, bandeira2015laplacian}, it is suggested information-theoretically exact recovery with high probability is impossible for synchronization over $\mathbb{Z}_2$ if $\sigma > \sqrt{n/(2-\varepsilon) \log n}$. We can use this result to show $\mathcal{O}(\sqrt{n/\log n})$ is necessary for a nontrivial $\ell_\infty$ error by putting a uniform prior on $\{-1, 1\}^n$. } 

It is important to state that the eigenvector estimator  mentioned above is order-wise as good an estimator as the MLE, in that it satisfies the same error bounds as in Theorem~\ref{thm:mainbound} up to constants. From the perspective of optimization, the main merit of Theorems 1 and 2 is that they rigorously explain the empirically observed tractability of (P) despite non-convexity. 

\subsection*{The difficulty: statistical dependence}

As will be argued momentarily, the main difficulty in the analysis is proving a sharp bound for $\| W \hat{x} \|_\infty$, which involves two dependent random quantities: the noise matrix $W$ and a solution $\hat x$ of~\eqref{eq:P}, which is a nontrivial function of $W$. While in the $\ell_2$-norm the simple bound $\| W \hat{x} \|_2 \leq \|W\|_2 \|\hat x\|_2$ is sharp,
no such simple argument is known to bound the $\ell_\infty$-norm. The need to study perturbations in $\ell_\infty$-norm appears inescapable, as it arises from the entry-wise constraints of~\eqref{eq:P} and the aim to control $\|\hat x - z\|_\infty$ as well as $\|\hat x - z\|_2$.



This issue has already been raised in~\citep[\S4]{bandeira2014tightness}, which focuses on the relaxation (\ref{eq:SDP}).
Specifically, in~\citep[eq.~(4.10)]{bandeira2014tightness}, it is shown that the relaxation is tight in particular if
\begin{align}
	n - 216\sigma^2 - 3\sigma \sqrt{n} - \sigma \|W\hat x\|_\infty > 0,
	\label{eq:originalSDPanalysis}
\end{align}
where $\hat x$ is an optimum of~\eqref{eq:P} which is then unique up to phase. From this expression, it is apparent that if $\sigma = \mathcal{O}(\sqrt{n/\log n})$, it only remains to show that $\|W\hat x\|_\infty = \mathcal{O}(\sqrt{n \log n})$ to conclude that solving~\eqref{eq:SDP} is equivalent to solving~\eqref{eq:P}. This reduces the task to that of carefully bounding this scalar, random variable:
\begin{align}
\|W \hat x\|_\infty & = \max_{k \in [n]} |w_k^* \hat x|,
\label{eq:Wxhatinfty}
\end{align}
where $w_1, \ldots, w_n \in \Cn$ are the columns of the random noise matrix $W$.
If $W$ and $\hat x$ were statistically independent, this would be bounded with high probability by $\mathcal{O}(\sqrt{n\log n})$, as desired. Indeed, since the vector $\hat x$ contains only phases and since the Gaussian distribution is isotropic (the distribution is invariant under rotation in the complex plane), $w_k^*\hat x$ would be distributed identically to a sum of $n-1$ independent standard complex Gaussians. The modulus of such a variable concentrates close to $\sqrt{n}$. Taking the maximum over $k \in [n]$ incurs an additional $\mathcal{O}( \sqrt{\log n})$ factor. 

Unfortunately,
the intricate dependence between $W$ and $\hat x$ has not been satisfactorily resolved in previous work, where only suboptimal bounds have been produced for $\|W\hat x\|_\infty$, eventually leading to suboptimal bounds on the acceptable noise levels $\sigma$~\citep[eq.~(4.11)]{bandeira2014tightness}\citep[Lemma~12]{boumal2016nonconvexphase}\citep[Proof of Thm.~2]{liu2016statistical}.

As a key step to overcome this difficulty, we (theoretically) introduce auxiliary problems to transform the question of controlling $\|W\hat x\|_\infty$ into one about the sensitivity of the optimum $\hat x$ to perturbations of the data
$C$. This is outlined next.


\subsection*{Introducing auxiliary problems to reduce dependence}\label{sec::introAux}

%
Since the main concern in controlling $\|W\hat x\|_\infty$ is the statistical dependence between $W$ and $\hat x$, we introduce $n$ new optimization problems of the form~\eqref{eq:P}, where, for each value of $m$ in $[n]$, the cost matrix $C$ is replaced by
\begin{align}
C^{(m)} & = zz^* + \sigma W^{(m)}, & & \textrm{ with } & W^{(m)}_{k\ell} & = W_{k\ell} \mathbf{1}_{\{k \neq m\}} \mathbf{1}_{\{\ell \neq m\}},
\label{eq:Cm}
\end{align}
where $\mathbf{1}$ is the indicator function. In other terms, $W^{(m)}$ is $W$ with the $m$th row and column set to 0, so that $C^{(m)}$ is statistically independent from $w_m$.
As a result, a global optimum $\hat x^{(m)}$ of~\eqref{eq:P} with $C$ set to $C^{(m)}$ is also independent from $w_m$. This usefully informs the following observation, where the global phases of $\hat x$ and $\hat x^{(m)}$ are chosen so that $\hat x^* \hat x^{(m)} = |\hat x^* \hat x^{(m)}|$:
\begin{align}
|(W\hat x)_m| = |w_m^* \hat x| & \leq |w_m^* \hat x^{(m)}| + |w_m^*(\hat x - \hat x^{(m)})| \nonumber\\
& \leq |w_m^*\hat x^{(m)}| + \|w_m\|_2 \|\hat x - \hat x^{(m)}\|_2. \label{eq:inftyboundsplit}
\end{align}
Crucially, independence of $w_m$ and $\hat x^{(m)}$ implies the first term is $\mathcal{O}(\sqrt{n\log n})$ with high probability, by the argument laid out after eq.~\eqref{eq:Wxhatinfty}. In the second term, a standard concentration argument shows $\|w_m \|_2 = \mathcal{O}(\sqrt{n})$ with high probability---see Section~\ref{sec:eigen}. Hence, to control $\|W\hat x\|_\infty$, it is sufficient to show that, with high probability for all $m$, the solutions $\hat x$ and $\hat x^{(m)}$ are within distance $\mathcal{O}(1)$ of each other, in the $\ell_2$ sense.

This claim about the proximity of $\hat x$ and $\hat x^{(m)}$ turns out to be a delicate statement about the sensitivity of the global optimum of~\eqref{eq:P} to perturbations of only measurements which involve the $m$th phase, $z_m$. To establish it, we need precise control of the properties of the optima of~\eqref{eq:P}. To this end, we develop a strategy to track the properties of sequences which converge to $\hat x$ as well as to $\hat x^{(m)}$ for each $m$.

To ease further discussion about $\ell_2$ distances up to phase, consider the following distance-like function:
\begin{align*}
	d_2(x, y) & := \min_{\theta\in\reals} \| xe^{i\theta} - y \|_2 , \qquad x,y \in \mathbb{C}^n.
\end{align*}
Restricted to complex vectors of given $\ell_2$-norm or to complex vectors with unit-modulus entries, $d_2$ is a true metric on the quotient space induced by the equivalence relation $\sim$:
\begin{align}
	x \sim y \iff \exists \, \theta : x = y e^{i\theta}.
	\label{eq:sim}
\end{align}
Thus, $d_2$ is appropriate as a distance between estimators for~\eqref{eq:P} and as a distance between candidate eigenvectors, being invariant under global phase shifts. Moreover, the quotient space is a complete metric space with $d_2$. More details will follow.

Coming back to our problem, the core argument is an analysis of
recursive error bounds of GPM, and this analysis leads to the proof that all iterates stay in $\calN := \calN_1 \cap \calN_2$, where
\begin{align}
	\calN_1 & = \{ x \in \mathbb{C}^n:  \| W x \|_\infty \le \kappa_2\sqrt{n \log n} \}, \label{eq:None}\\
	\calN_2 & = \{ x \in \mathbb{C}^n: d_2(x, z) \le \kappa_3\sqrt{n} \}, \label{eq:Ntwo}
\end{align}
and $\kappa_2,\kappa_3>0$ are some constants (determined in Section \ref{sec:phaseSyn}). 
On one hand, we show that, with high probability, the nonlinear mapping $\calT x = \mathcal{P}(Cx)$ iterated by GPM is Lipschitz continuous over $\calN$ with constant $\rho \in (0, 1)$. On the other hand, we show that, with high probability, all iterates of GPM are in $\calN$. Together, these two properties imply that, with high probability, $\calT$ is a contraction mapping over the set of iterates of GPM. By a completeness argument, this implies that the sequence of iterates of GPM converges in $\calN$. The roadmap of our proof is the following:
\begin{enumerate}
	\item $x^0 \in \calN$ (this requires developing new $\ell_\infty$ bounds for eigenvector perturbation---see Theorem \ref{thm::mainEig2}),
	\item  $x^t \in \calN \implies x^{t+1} \in \calN$ (this is done in two stages,\footnote{To be precise, for large $t$, we consider a slightly larger $\calN$ (the constants in (\ref{eq:None}) and (\ref{eq:Ntwo}) are larger), and prove that all $x^t$ stay in this larger region. This is a technical issue which does not affect the overall plan.} for small and large $t$---see Theorems~\ref{thm::indct} and~\ref{thm::finCnvg}),
	\item $\calT $ is $\rho$-Lipschitz with $\rho < 1$ on $\calN$ with respect to $d_2$ (by completeness, this implies $\lim_{t\to\infty} x^t = x^\infty \in \calN$---see Lemma \ref{lem::ctrT}), and
	\item any fixed point $x^\infty$ of $\mathcal{T}$ in $\calN$ is a global optimum of~\eqref{eq:P} (see Lemma \ref{lem:certificate}).
\end{enumerate}
On top of securing results about GPM, this will imply that~\eqref{eq:P} admits a solution $\hat x = x^\infty$ which is in $\calN$ and hence, a fortiori, satisfies $\| W \hat x \|_\infty \le \kappa_2\sqrt{n \log n}$, yielding the announced results about the SDP relaxation as per~\eqref{eq:originalSDPanalysis}.

As hinted above, we follow this reasoning not only for the sequence $x^t$ which is expected to converge to $\hat x$, but also for auxiliary sequences $x^{t,m}$ expected to converge to $\hat x^{(m)}$. It is only through exploitation of the strong links between these sequences and reduction in statistical dependence they offer that we are able to go through with the proof program above.

Note that $\calT$ might not be a contraction mapping on all of $\calN$ since we do not show that $\calT(\calN) \subset \calN$. Nevertheless, $\calT$ is a contraction on the iterates, which is sufficient for our purpose; henceforth, we say the mapping has the \emph{local contraction property}.

We remark that, in the study of high-dimensional $M$-estimation \cite{bean2013optimal}, the idea of introducing auxiliary problems (and associated optimizers) is also used to tackle dependence, and it yields powerful analysis. While sharing similarity with that approach, our analysis relies on studying $n$ auxiliary sequences of iterates, as will be discussed soon---also see Figure \ref{fig:contraction}.

As a necessary and useful warm-up, we first focus on the task of showing that $x^0$ (a leading eigenvector of $C$) is in $\calN$, via analysis of the related $x^{0,m}$ (leading eigenvectors of $C^{(m)}$). This requires sharp bounds for $d_2(x^0, x^{0,m})$. The outcome of this analysis is an eigenvector perturbation bound in the $\ell_\infty$-norm, which is another motivation for the introduction of auxiliary problems.

\subsection*{First analysis: an $\ell_\infty$ perturbation bound for eigenvectors}



As the initializer of Algorithm \ref{algo:GPM}, the leading eigenvector of $C$ has several good properties necessary for analysis, and we will discuss them in depth in Section \ref{sec:eigen}. Theorems and lemmas in this direction are stated separately and proved first, because
their proof is illustrative of the techniques deployed to prove results about~\eqref{eq:P}.
Most notably, we prove a sharp $\ell_\infty$ perturbation bound for leading eigenvectors.

\begin{theorem}\label{thm::mainEig}
If $\sigma = \mathcal{O}(\sqrt{n/\log n})$, then, with high probability for large $n$, a leading eigenvector $x^0$ of the data matrix $C$~\eqref{eq:C} scaled such that $\|x^0\|_2 = \sqrt{n}$ satisfies 
\begin{align}
	\| x^0 - z \|_{\infty} = \mathcal{O}( \sigma \sqrt{\log n / n}),
	\label{ineqn::mainEig}
\end{align}
where the global phase of $x^0$ is suitably chosen (e.g., such that $z^*x^0 = |z^*x^0|$.)
\end{theorem}

The crux of the proof lies in a sharp bound on $d_2(x^0, x^{0,m})$.
This is obtained by using (a suitable version of) the Davis--Kahan theorem (Lemma~\ref{lem::ptb}): when $\sigma \ll \sqrt{n}$, 
\begin{align*}
	d_2(x^0, x^{0,m}) & = \mathcal{O} \left( \frac{\sigma \| (W - W^{(m)})  x^{0,m} \|_2}{n} \right) = \mathcal{O}(\sigma \sqrt{\log n/n}), \quad \text{whereas} \\
	d_2(x^0, z) & =  \mathcal{O} \left( \frac{\sigma \| W z \|_2 }{n} \right) = \mathcal{O}(\sigma). 
\end{align*}
To reach the first conclusion, we view $x^0$ as the perturbed version of $x^{0,m}$ due to perturbation $\sigma (W - W^{(m)})$. Note that $W - W^{(m)}$ has nonzero entries only in the $m$th row and $m$th column, and they are independent of $x^{0,m}$. Compared to the full perturbation $\sigma W$ which perturbs the eigenvector $z$ to $x^0$, the matrix $\sigma (W - W^{(m)})$ results in a much smaller $d_2$ distance between $x^{0,m}$ and $x^0$. Notice that, as will be detailed later, these results combined with the reasoning of~\eqref{eq:inftyboundsplit} imply that $x^0$ is in $\calN$, as desired.

Comparing Theorem~\ref{thm::mainEig} to Theorem~\ref{thm:mainbound} readily shows that the eigenvector $x^0$ is an excellent estimator for $z$ (up to the fact that its entries are not necessarily unit-modulus, which can be easily corrected---see Theorem \ref{thm::mainEig2}). 
Further efforts in this paper are dedicated to characterizing the performance and tractability of the MLE $\hat x$. 


\subsection*{Analysis of iterations: tracking $n$ auxiliary sequences}
While analyzing the eigenvector $x^0$ is relatively straightforward, the optima of~\eqref{eq:P} are more difficult to tame due to the unit-modulus constraints. As hinted above, the novel idea we develop in this paper is to track the sequences $\{x^{t,m}\}_{t=0}^\infty$ produced by Algorithm~\ref{algo:GPM} with inputs $C^{(m)}$~\eqref{eq:Cm} instead of $C$, for each $m \in [n]$. These auxiliary sequences---which only serve for the analysis and are not (and could not be) computed in practice---enjoy the crucial proximity property desired in the previous subsection---see Figure~\ref{fig:contraction}.

\begin{figure}[t]
	\centering
	\includegraphics[width=4in]{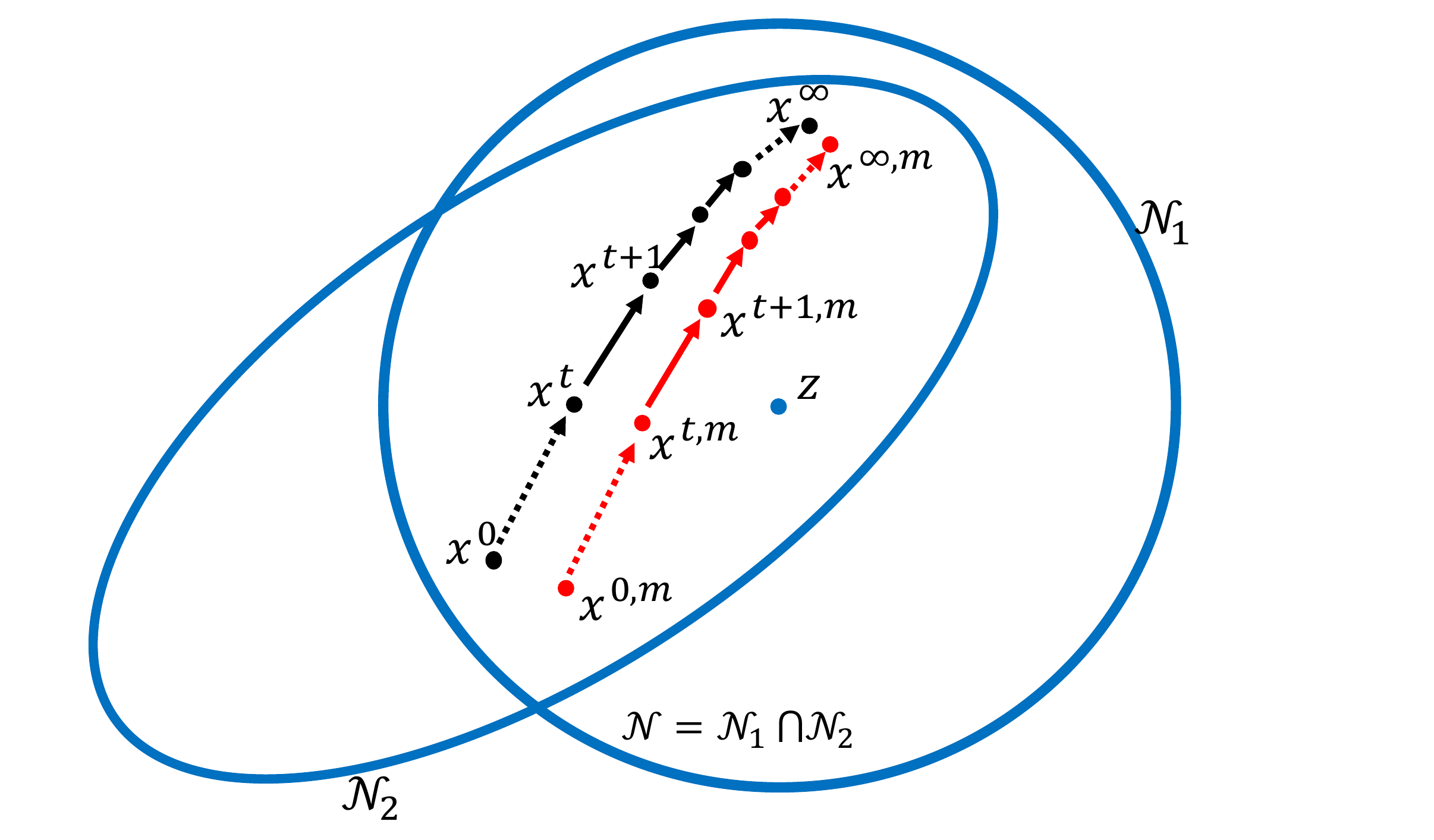}
	\caption{Sequence $\{x^t\}_{t=0}^\infty$ (in black) produced by Algorithm~\ref{algo:GPM}, and $n$ auxiliary sequences $\{x^{t,m}\}_{t=0}^\infty$ (in red) produced (conceptually) by Algorithm~\ref{algo:GPM} with modified inputs. Crucial properties along the paths: (i) proximity: $x^t$ and $x^{t,m}$ stay close; (ii) local contraction: $x^t$ and $x^{t,m}$ remain in the contraction region $\calN$ with high probability and converge in it.}
	\label{fig:contraction}
\end{figure}


Indeed, we will show by induction that there exist absolute constants $\kappa_1, \kappa_2, \kappa_3$ such that, for all $m$ and for $t = 0, 1, 2, \ldots$, 
\begin{samepage}
\begin{enumerate}
	\item $d_2(x^{t} , x^{t,m}) \leq \kappa_1$, \qquad \qquad \qquad \qquad \qquad \qquad ~~ proximity property
	\item \ntikzmark{L}{$x^t \in \calN_1 = \{ x\in \mathbb{C}^n: \|Wx\|_\infty \leq \kappa_2 \sqrt{n\log n}\}, $ }
	\item \ntikzmark{O}{$x^t \in \calN_2 = \{ x\in \mathbb{C}^n: d_2(x,z) \leq \kappa_3 \sqrt{n} \}$.}
	\makebrace{L}{O}{contraction region}
\end{enumerate}
\end{samepage}
The proximity and local contraction properties\footnote{Formally, we only prove $x^t \in \calN$; but proximity implies $x^{t,m}$ is also in a (slightly larger) contraction region, with different constants $\kappa_2, \kappa _3$.} are both crucial and complementary for the analysis: the proximity property allows to control $\ell_\infty$ quantities in the presence of the random matrix $W$ despite statistical dependence (as shown in~\eqref{eq:inftyboundsplit}), and the local contraction property is used to establish~\eqref{recurProx} below, making sure $d_2(x^{t} , x^{t,m})$ remains small. 

For the high-level idea, consider the nonlinear operators $\calT$ and $\calT^{(m)}$ implicitly defined by Algorithm~\ref{algo:GPM} so that $x^{t+1} = \mathcal{T} x^t$ and $x^{t+1,m} = \mathcal{T}^{(m)} x^{t,m}$. If we can show that $\calT$ is $\rho$-Lipschitz with constant $\rho\in(0,1)$ with respect to $d_2$, then a recursive error bound follows:
\begin{align}
	d_2(x^{t+1}, x^{t+1,m}) & = d_2(\calT x^{t}, \calT^{(m)} x^{t,m}) \nonumber\\
	 & \le d_2(\calT x^{t}, \calT x^{t,m}) + d_2(\calT x^{t, m}, \calT^{(m)} x^{t,m}) \nonumber\\
	&\le \rho \cdot d_2(x^{t}, x^{t,m}) + \textrm{discrepancy error}.
	\label{recurProx}
\end{align}
This ensures $d_2(x^{t} , x^{t,m})$ does not accumulate with $t$, provided the discrepancy error---which is caused by the difference between $\mathcal{T}$ and $\mathcal{T}^{(m)}$---is small enough.
This is assured with high probability, because $C - C^{(m)}$ is independent of $x^{t,m}$, causing the discrepancy error to be $\mathcal{O}(\sigma\sqrt{\log n / n})$---considerably smaller than $d_2(x^{t+1}, z) = \mathcal{O}(\sigma)$. In spirit, this is the same argument as in the analysis of the eigenvector estimator.

The above 
recursive error bound hinges on the other important property, that is, $x^t$ staying in the contraction region. Crucially, to establish $x^t \in \calN_1$, we need a tight bound on $\| Wx^t \|_\infty$. Fortunately, we have seen how to control this quantity in~\eqref{eq:inftyboundsplit}: for any $m \in [n]$, 
\begin{equation}
	|(W x^t)_m| \le |w_m^* x^{t,m}| + \|w_m\|_2 \cdot  d_2(x^t , x^{t,m}). \label{eq:inftyboundsplit2}
\end{equation}
This, in turn, requires a proximity result for $d_2(x^t , x^{t,m})$. This insight naturally motivates an analysis of each iteration by induction.

There are two technical issues we briefly address before ending this introduction with a remark.

The first issue concerns the probabilistic argument in the proof. In (\ref{recurProx}) and (\ref{eq:inftyboundsplit2}), we invoke concentration inequalities to obtain tight bounds. However, since there is a (small) probability that such inequalities fail, we cannot use union bounds for $t=1,2,\ldots$, which is infinite. To overcome this obstacle, we use concentration inequalities only for the first $T = \mathcal{O}(n^2)$ iterations, and resort to a deterministic analysis for iterations $t > T$. The critical observation is that, $d_2(x^{t+1}, x^t)$ decays exponentially for $t \le T$ due to contraction, so the amount of update is tiny after $T$ iterations. Using another inductive argument, we can secure exponential decay for $t > T$ as well. The rationale is that we already established good properties about $x^T$, and $d_2(x^t, x^T)$ is tiny for $t > T$, so we can easily relate $x^t$ to $x^T$ and show $x^t$ also has good properties. Essentially, $x^t$ remains in a contraction region with slightly larger constants.

The second issue is identifying the limit $x^\infty = \lim_{t\to \infty} x^t$ with a solution of~\eqref{eq:P}. 
We will verify the optimality and uniqueness (up to phase) of $x^\infty$
via
a known dual optimality certificate $S = \Re\{\ddiag(Cx^\infty (x^\infty)^*)\} - C$ \citep{bandeira2014tightness}.

We close with a remark about the initializer $x^0$ (and $x^{0,m}$). Algorithm~\ref{algo:GPM} uses the leading eigenvector of $C$ for initialization, and our analysis relies on $\ell_\infty$ perturbation bounds to verify the base case of the induction. However, we point out that, even in the absence of such perturbation results, we could set $x^0 = x^{0,m} = z$ in theory, deduce that $\hat x \in \calN$ and thus prove Theorem~\ref{thm:mainSDP} about the SDP relaxation. In other words, even if we leave out the discussion about the initializer altogether, there is still enough material to secure tightness of the SDP relaxation. The proof augmented with the analysis of the eigenvector perturbation has the advantage of also providing a statement about Algorithm~\ref{algo:GPM} which is actually runnable in practice.

\section{Main results}\label{sec:mainres}

In this section we will state our main theorems formally.
The assumption on random noise $W$ will also be relaxed to a broader class of random matrices. To begin with, let us first clarify the ``up to phase'' statements in Section~\ref{sec:intro}.

\paragraph*{The quotient space}
For any $\theta \in \mathbb{R}$, whether the true phases are $z$ or $ze^{i\theta}$ does not affect the measurements $C$~\eqref{eq:C}. As a result, the available data are insufficient to distinguish $z$ from $ze^{i\theta}$. Clearly the program~\eqref{eq:P} is invariant to global phase shifts as well. It thus makes sense to ignore the global phase in defining distances between estimators. A reasonable notion of $\ell_2$ error then becomes
\begin{align}
	d_2(x, y) & = \min_{\theta\in\reals} \| xe^{i\theta} - y \|_2 = \sqrt{\|x\|_2^2 + \|y\|_2^2-2|x^*y|},
	\label{eq:dtwo}
\end{align}
where the optimal phase $e^{i \theta}$ is the phase (Arg) of $x^* y$. 
Similarly, a notion of $\ell_\infty$ error can be defined:
\begin{align}
	d_\infty(x, z) & =   \min_{\theta\in\reals} \| xe^{i\theta} - z \|_\infty.
	\label{eq:dinf}
\end{align}
Formally, one can partition all points in $\mathbb{C}^n$ into equivalence classes via the equivalence relation $\sim$~\eqref{eq:sim}. The resulting quotient space $\Cn /\! \sim$ contains the equivalence classes $[x] = \{ xe^{i\theta} : \theta \in \reals \}$ for all $x \in \Cn$. Specifically, the feasible set of~\eqref{eq:P},
\begin{equation}\label{def:Cn1}
\Cn_1 := \{ x \in \Cn : |x_1| = \cdots = |x_n| = 1 \},
\end{equation} 
reduces to $\Cn_1 / \! \sim$ under this equivalence relation. It is easily verified that $d_2$ defines a distance on $\Cn_1 /\! \sim$. In particular, it satisfies the triangular inequality (where $d_2(x, y)$ is understood to mean $d_2([x], [y])$):
\begin{align*}
	\forall x,y,z \in \Cn_1, \quad d_2(x, y) & = \min_{\theta_1, \theta_2} \|xe^{i\theta_1} - z + z - ye^{i\theta_2}\|_2 \\ & \leq \min_{\theta_1} \|xe^{i\theta_1} - z\|_2 + \min_{\theta_2} \|z - ye^{i\theta_2}\|_2 = d_2(x, z) + d_2(z, y).
\end{align*}
Moreover, $\Cn_1 / \! \sim$ is a complete metric space under $d_2$ (see Theorem \ref{thm::finCnvg}). Similarly, $d_\infty$ is also a distance.
As will be shown, the sequence $\{x^t\}_{t=1}^\infty$ described in Section \ref{sec:intro} satisfies the local contraction property (see~\eqref{recurProx}) on the metric space $(\Cn_1 /\! \sim, d_2)$, hence converges to a fixed point which is exactly $\hat{x}$ (understood as $[\hat{x}]$).

\paragraph*{The noise matrix} In Section \ref{sec:intro} we assume that $W$ has independent standard complex Gaussian variables above its diagonal. However, this restricted assumption is only for expository convenience, and can be relaxed to the class of Hermitian Wigner matrices with sub-gaussian entries. Statements about Algorithm~\ref{algo:GPM} and about tightness of the SDP relaxation continue to hold, although of course the solution of~\eqref{eq:P} now no longer necessarily corresponds to the MLE.

The class of sub-gaussian variables subsumes Gaussian variables, but has one defining feature similar to Gaussian variables, that is, the tail probability decaying no slower than Gaussian variables. In our model, each entry of $W$ satisfies the tail bound
\begin{equation}\label{eq:subdef}
\mathbb{P}( | \xi | > t ) \le \exp( 1 - t^2/K^2)
\end{equation}
for both real and imaginary parts, where $K > 0$ is an absolute constant. Formally, we assume the Hermitian matrix $W$ satisfies the following: $\{ \textrm{Re}(W_{k\ell}), \textrm{Im}(W_{k\ell}) \}_{k < \ell}$ are jointly independent, have zero mean, and satisfy the sub-gaussian tail bound (\ref{eq:subdef}); the diagonal elements are zero, and $W_{\ell k} = \overline{W_{k \ell}}$ for any $k < \ell$. Note there are equivalent definitions of sub-gaussian variables (up to constants) \citep{Ver10}.

This random model is a much richer class of noise matrices, containing the Gaussian model introduced in Section \ref{sec:intro} as a special case. Each random variable in $\{ \textrm{Re}(W_{k\ell}), \textrm{Im}(W_{k\ell}) \}_{k \le \ell}$ can be, for example, a symmetric Bernoulli variable, any other centered and bounded variable, or simply zero.

\paragraph*{SDP approach}

The SDP approach tries to solve~(\ref{eq:P}) via its convex semidefinite relaxation~(\ref{eq:SDP}).
It is a relaxation of~\eqref{eq:P} in the following sense. For any feasible $x\in\Cn$, the corresponding matrix $X = xx^*$ is feasible for~\eqref{eq:SDP}. Likewise, any feasible matrix $X$ of rank~1 can be factored as $X = xx^*$ such that $x$ is feasible for~\eqref{eq:P}. Thus, the relaxation consists in allowing solutions of rank more than 1 in~\eqref{eq:SDP}. Consequently, if~\eqref{eq:SDP} admits a solution of rank~1, $X = \hat x \hat x^*$, then the corresponding $\hat x$ is a global optimum for~\eqref{eq:P}. Furthermore, if the rank-1 solution of~\eqref{eq:SDP} is unique, it can be recovered in polynomial time. For this reason, the regime of interest is one where~\eqref{eq:SDP} admits a unique solution of rank~1.

The following theorem---a statement of Theorem~\ref{thm:mainSDP} which holds in the broadened noise model---closes the gap in previous papers \citep{bandeira2014open, bandeira2014tightness, liu2016statistical, boumal2016nonconvexphase}.
\begin{theorem}\label{thm:mainSDP2}
	There exists an absolute constant $c_0>0$ such that, if $\sigma \le c_0\sqrt{n / \log n}$, then, with probability $1- \mathcal{O}(n^{-2})$, the semidefinite program~\eqref{eq:SDP} admits a unique solution $\hat x \hat x^*$ where $\hat x$ is the unique global optimum of~\eqref{eq:P} up to phase.
\end{theorem}
Note that the exponent~2 in the failure probability $\mathcal{O}(n^{-2})$ can be replaced by any positive numerical constant, only affecting other absolute constants in the theorem (and all other theorems).

\paragraph*{GPM approach}

The generalized power method (Algorithm \ref{algo:GPM}) is an iterative algorithm similar to the classical power method, but instead of projecting vectors onto a sphere after matrix-vector multiplication $Cx^{t-1}$, it extracts the phases from $Cx^{t-1}$, which is an entry-wise projection.
It is much faster than SDP, and converges linearly to a limit, which is the optimum $\hat{x}$ up to phase (optimality is stated in Theorem \ref{thm:mainbound2}). The next theorem is a precise version of Theorem~\ref{thm:mainGPM}.
\begin{theorem}\label{thm:mainGPM2}
There exists an absolute constant $c_0>0$ such that, if $\sigma \le c_0\sqrt{n / \log n},$ then, with probability $1- \mathcal{O}(n^{-2})$, the sequence $\{ x^t \}$ produced by Algorithm \ref{algo:GPM} has a linear convergence rate to some $\hat x \in \Cn_1$ (up to phase):
\begin{equation*}
d_2( x^t , \hat{x} ) \le 2^{2-t} \sqrt{n}, \qquad t = 0, 1, 2\ldots
\end{equation*}
\end{theorem}
The proof is based on induction: in each iteration, we will establish the proximity property and the contraction property for $x^t$. The proof is simply a rigorous justification of the heuristics we discussed in Section \ref{sec:intro}. It also leads to $\ell_2$ and $\ell_\infty$ error bounds for $\hat{x}$. The following theorem is a precise version of Theorem \ref{thm:mainbound}.

\begin{theorem}\label{thm:mainbound2} 
Under the same conditions and notation as in Theorem \ref{thm:mainGPM2}, with probability $1- \mathcal{O}(n^{-2})$, $\hat{x}$ is the unique optimum of~\eqref{eq:P} up to phase. If we choose the global phase of $\hat x$ such that
$\hat x^* z = |\hat x^* z|$, then
\begin{align*}
	d_2( \hat{x} , z ) & = \|\hat{x} - z \|_2 \le C_0 \sigma, \textrm{ and } & d_{\infty}( \hat{x} , z ) & \le \|\hat{x} - z \|_\infty \le C_0 \sigma \sqrt{\log n/n},
\end{align*}
where $C_0 > 0$ is an absolute constant.
\end{theorem}

\paragraph*{Eigenvector estimator}

We denote, henceforth, the leading eigenvector of $C$ by $\tilde{x}$, and similarly the leading eigenvector of $C^{(m)}$ by $\tilde{x}^{(m)}$. Note that $\tilde{x}$ and $x^0$ (similarly $\tilde{x}^{(m)}$ and $x^{0,m}$) are identical.\footnote{We use notation $\tilde{x}$ (similarly $\tilde{x}^{(m)}$) to emphasize the leading eigenvector as an estimator, as opposed to merely an initializer.}  
We highlight the significance of the eigenvector estimator in the following theorem, which is a precise version of Theorem \ref{thm::mainEig}.
\begin{theorem}\label{thm::mainEig2} 
Let $\tilde x$ be a leading eigenvector of $C$ with
scale and global phase chosen such that $\|\tilde x\|_2 = \sqrt{n}$ and $\tilde x^* z = |\tilde x^* z|$.
There exists an absolute constant $c_0'>0$ such that, if $\sigma < c_0' \sqrt{n/\log n}$, then, with probability $1 - \mathcal{O}(n^{-2})$,
\begin{align*}
	d_2( \tilde{x},  z ) & = \|  \tilde{x} - z \|_2 \le C_0'\sigma, \textrm{ and }&
	d_\infty( \tilde{x} , z ) & \le \|  \tilde{x} - z \|_\infty \le C_0' \sigma \sqrt{\log n /n},
\end{align*}
where $C_0' > 0$ is some absolute constant. Moreover, the projected leading eigenvector, namely, $\mathcal{P} \tilde{x} $, satisfies the same bounds with $C_0'$ replaced by $2C_0'$.
\end{theorem}

The eigenvector estimator has been studied extensively in recent years, prominently in the statistics literature~\citep{Pau07,JohLu12}, under the \emph{spiked covariance model}. While the perturbation $\tilde{x} - z$ is usually studied under $\ell_2$ norms, the $\ell_{\infty}$ norm received much less attention. A recent $\ell_{\infty}$ perturbation result appeared in \cite{FanWanZho16}, but it is a deterministic bound and would produce a suboptimal result here.

\section{Proof organization for eigenvector perturbations}\label{sec:eigen}

We begin with some concentration lemmas, which will also be useful in Section \ref{sec:phaseSyn}. Recall the definition of $W^{(m)}$~\eqref{eq:Cm}. We also define $\Delta W^{(m)} := W - W^{(m)}$, which has nonzero entries only in the $m$th row and $m$th column, given by $w_m$.

\subsection*{Concentration lemmas}
The first concentration result is standard and is a direct consequence of, for example, Proposition 2.4 in~\cite{RudVer10}.
\begin{lemma}\label{lem::conctr2}
With probability $1 - \mathcal{O}(n^{-2})$, the following holds for any $m\in [n]$:
\begin{equation}\label{ineqn::W}
\| W \|_2 \le C_2'\sqrt{n}, \quad \| W^{(m)} \|_2 \le C_2' \sqrt{n}, \quad \| \Delta W^{(m)} \|_2 \le C_2' \sqrt{n}, \quad \| w_m \|_2 \le C_2' \sqrt{n},
\end{equation}
where $C_2' > 0$ is an absolute constant.
\end{lemma}

Let $\mathbb{S}^{n-1}$ be the set of unit vectors in $\mathbb{C}^n$. Suppose for each $m \in [n]$ we have a finite (random) set $\mathcal{U}^{(m)} \subset \sqrt{n}\,\mathbb{S}^{n-1}$ whose elements are independent of $\Delta W^{(m)}$, and the cardinality of $\mathcal{U}^{(m)}$ is not random. Concentration inequalities enable us to bound $\| \Delta W^{(m)} u \|_2$ uniformly over all $u \in \mathcal{U}^{(m)}$ with high probability. We state this formally in the next lemma. 

\begin{lemma}\label{lem::conctr1}
Suppose $|\mathcal{U}^{(m)}| \le 3n^2$ for all $m \in [n]$. Let $u_m$ be the $m$th entry of a vector $u \in \mathcal{U}^{(m)}$ and denote $M_m = \max_{u \in\mathcal{U}^{(m)}} | u_m |$. Then, with probability $1-\mathcal{O}(n^{-2})$, 
\begin{equation}\label{ineqn::conctr1}
\max_{u \in \mathcal{U}^{(m)} } \| \Delta W^{(m)} u \|_2 \le C_1' \sqrt{ n\log n } + C_1' \sqrt{n}\, M_m, \quad \forall \, m \in [n],
\end{equation}
where $C_1'>0$ is an absolute constant. In particular, we can choose $C_1'>0$ such that, if $\mathcal{U}^{(m)} \subset \Cn_1$ (defined in~\eqref{def:Cn1}), 
then, with probability $1-\mathcal{O}(n^{-2})$, 
\begin{equation}\label{ineqn::conctr1new}
\max_{m \in [n]} \max_{u \in \mathcal{U}^{(m)} } \| \Delta W^{(m)} u \|_2 \le C_1'\sqrt{ n\log n }.
\end{equation}
\end{lemma}

A straightforward application of Hoeffding's inequality for sub-gaussian variables~\citep{Ver10} shows $\max_{u \in \mathcal{U}^{(m)} } | w_m ^* u | = \mathcal{O}(\sqrt{n \log n})$  with probability $1-\mathcal{O}(n^{-2})$. Lemma~\ref{lem::conctr1} is more general, because $| w_m ^* u | \le \| \Delta W^{(m)} u \|_2$, and it will be useful in later proofs.

For the eigenvector problem, we will choose $\mathcal{U}^{(m)} = \{ \tilde{u}^{(m)} \}$ (a singleton) where $\tilde{u}^{(m)}$ is a leading eigenvector of $C^{(m)} = zz^* + \sigma W^{(m)}$ scaled to have norm $\sqrt{n}$. For problem~(\ref{eq:P}), for each $m \in [n]$, the set $\mathcal{U}^{(m)}$ will be $\{x^{t,m}\}_{t=0}^T$, namely,  
the first $T+1$ iterates of Algorithm~\ref{algo:GPM} with input $C^{(m)}$, where $T := 3n^2-1$. By construction, elements of the set $\mathcal{U}^{(m)}$ are independent of $\Delta W^{(m)}$.

\subsection*{Introducing auxiliary eigenvector problems}

As is well known, the leading eigenvectors of $C$ are the solutions to the following optimization problem (note that this problem is a relaxation of~\eqref{eq:P}):
\begin{equation}\label{prob:eig}
\begin{array}{rlc}\tag{\text{$\tilde{\text{P}}$}}
\underset{x \in \mathbb{C}^n}{\max} & x^* C x  &\\
\mathrm{s.t.} & \| x \|_2 = \sqrt{n}. &
\end{array}
\end{equation}
We aim to show that a solution $\tilde x$ of~\eqref{prob:eig} is close to $z$ in the sense of $d_\infty$. As before, the major difficulty of the analysis is obtaining a sharp bound on $\| W \tilde{x} \|_{\infty}$. This is apparent when we write $\lambda_1(C)$ for the leading eigenvalue of $C$ and use $C\tilde{x} = \lambda_1(C) \tilde{x}$ to obtain (choosing the global phase of $\tilde{x}$ such that $z ^* \tilde{x} = |z ^* \tilde{x}|$):
\begin{equation*}
|\tilde{x}_m - z_m| = \left|\frac{(C \tilde{x})_m}{\lambda_1(C)}  - z_m\right| \le \left|\frac{ |z ^* \tilde{x}|}{\lambda_1(C)} - 1 \right| + \frac{\sigma|(W\tilde{x})_m|}{\lambda_1(C)}, \quad \forall \, m \in[n].
\end{equation*}
While it is easy to analyze $\lambda_1(C)$ and $|z ^* \tilde{x}|$, bounding $\| W \tilde{x} \|_{\infty}$ requires more work.

For $m \in [n]$, let $\tilde{x}^{(m)}$ be the solution to an auxiliary problem~\eqref{prob:eig} in which $C$ is replaced by $C^{(m)} = zz^* + \sigma W^{(m)}$---thus, $\tilde{x}^{(m)}$ is equivalent to $x^{0,m}$. 
Following the same strategy as in~\eqref{eq:inftyboundsplit}, we can now split $W\tilde{x}$ into two terms and try to bound separately:
\begin{align}\label{ineqn::eigKey}
|(W\tilde{x})_m| = |w_m^* \tilde x| \le | w_m ^* \tilde{x}^{(m)} | + \| w_m \|_2 \cdot d_2( \tilde{x} , \tilde{x}^{(m)})
\end{align}
where $| w_m ^* \tilde{x}^{(m)} |$ is the dominant term, and can be easily bounded---see the paragraph below Lemma \ref{lem::conctr1}; and $\| w_m \|_2 \cdot d_2( \tilde{x} , \tilde{x}^{(m)}) = \mathcal{O}(\sqrt{n}\, d_2( \tilde{x} , \tilde{x}^{(m)}) )$ is the higher-order discrepancy error, which is the price we pay for replacing $\tilde{x}$ with $\tilde{x}^{(m)}$. 

The crucial point is that $d_2( \tilde{x} , \tilde{x}^{(m)}) = \mathcal{O}(\sigma \sqrt{\log n / n})$, which is much smaller than $d_2(\tilde{x}, z) = \mathcal{O}(\sigma)$. This is because the difference between $ \tilde{x}$ and $\tilde{x}^{(m)}$ results from a sparse perturbation $\Delta W^{(m)}$, whose effect on the leading eigenvector is small. This point is formalized in the next lemma, which follows from \citep{DavKah70}.\footnote{By Davis-Kahan Theorem and Weyl's inequality, $\sin \theta( \tilde{u} , u) \le  \| E u \|_2 / \sqrt{n}\,(\delta - \|E \|_2)$. Thus, the lemma follows from $d_2(  \tilde{u} , u )^2/n = 2 - 2| \tilde{u}^*  u | / n = 2 - 2\cos \theta( \tilde{u} , u) \le 2\sin^2 \theta( \tilde{u} ,u)$.}

\begin{lemma}[Davis--Kahan $\sin \Theta$ Theorem]\label{lem::ptb}
Suppose that $A,E \in \mathbb{C}^{n \times n}$ are Hermitian matrices, and $\tilde{A} = A + E$. Let $\delta := \lambda_1(A) - \lambda_2(A)$ be the gap between the top two eigenvalues of $A$, and $u, \tilde{u}$ be leading eigenvectors of $A$ and $\tilde{A}$ respectively, normalized such that $\|u\|_2 = \|\tilde{u}\|_2 = \sqrt{n}$.
If $\delta >  \|E \|_2$, then
\begin{equation}\label{ineqn::ptb}
d_2(  \tilde{u} , u ) \le \frac{\sqrt{2}\, \| E u \|_2}{\delta - \|E \|_2 }.
\end{equation}
\end{lemma}
The benefit of this perturbation result is pronounced when $E$ is a sparse random matrix: if we set $A = C^{(m)}, E = \sigma \Delta W^{(m)}$, then the numerator in (\ref{ineqn::ptb}) becomes $\sqrt{2}\, \sigma \| \Delta W^{(m)} \tilde x^{(m)} \|_2 = \mathcal{O}(\sigma \sqrt{n\log n })$ with high probability (by Lemma \ref{lem::conctr1} and a bound on $M_m$). If, however, we set $A = zz^*, E = \sigma W$, then the numerator is $\sqrt{2}\, \sigma \|W z \|_2 = \mathcal{O}(\sigma n)$ with high probability. This is why $d_2( \tilde{x} , \tilde{x}^{(m)})$ is so small and (\ref{ineqn::eigKey}) yields a tight bound.

We remark that in many later uses of perturbation results (e.g., \citep{Von07,RohChaYu11}), especially in statistics and theoretical computer science, it is common to invoke a variant of the Davis--Kahan theorem in which $\| E u \|_2$ is replaced by $\| E \|_2$ in (\ref{ineqn::ptb}), which would lead to a suboptimal result here. This is because $\| Eu \|_2 \ll \| E\|_2$ with high probability when $E$ is a random sparse matrix and $\|u\|_\infty$ is not large. Our analysis here is an example that shows the merit of using the more precise version of the Davis--Kahan theorem. 

\section{Proof organization for phase synchronization}\label{sec:phaseSyn}
We begin with some useful lemmas about the local contraction property.
These will prove useful to establish the desired properties of the iterates $x^t$ of Algorithm~\ref{algo:GPM}, by induction. These properties extend to the limit $x^\infty = \lim_{t \to \infty} x^t$ by continuity. 
Finally, we will use
a known optimality certificate for~\eqref{eq:SDP} to validate $x^\infty$.

\subsection*{Local contraction lemmas}
First let us denote the rescaled matrix $C/n$ by $\mathcal{L}$, which can also be viewed as a linear operator in $\mathbb{R}^n$:
\begin{equation*}
\mathcal{L}x := \frac{1}{n}Cx = \frac{z^*x}{n}z + \frac{\sigma}{n}Wx, \qquad \forall \, x \in \mathbb{C}^n.
\end{equation*}
This is a linear combination of $z$ (signal) and $Wx$ (noise).
When $\sigma\|W\|_2$ is small compared to $n$, $\mathcal{L}$ is Lipschitz continuous on the sphere around $z$, with respect to $d_2$.

\begin{lemma}\label{lem::ctr}
Suppose $\varepsilon \in (0,1/2)$, and $x,y \in \sqrt{n}\,\mathbb{S}^{n-1}$, with $ d_2( x , z ) \le \varepsilon \sqrt{n}$, $d_2( y , z ) \le \varepsilon \sqrt{n}$. Then,
\begin{equation}\label{ineqn::ctr}
d_2( \mathcal{L}x , \mathcal{L} y ) \le \left(6\varepsilon + \frac{\sigma \| W \|_2}{n} \right) d_2(x,y). 
\end{equation}
\end{lemma}
This lemma is instrumental in establishing the key local contraction property (Lemma~\ref{lem::ctrT}). It is related to the contraction mapping theorem, in which an iteratively defined sequence converges to a fixed point. We could use this lemma to easily show (using \citep[Proof of Thm.~2]{liu2016statistical} for example) that the normalized version $\bar{\mathcal{L}} x := \mathcal{L} x / \| \mathcal{L} x \|_2$ is a contraction mapping in a neighborhood of $z$ on $\sqrt{n}\,\mathbb{S}^{n-1}$---$\bar{\mathcal{L}}$ is the power method operator.

However, our problem is more complicated due to the unit-modulus constraints in~\eqref{eq:P} which call for the entry-wise operator $\mathcal{P}(\cdot)$ in Algorithm~\ref{algo:GPM}. Consequently, an analysis of Algorithm~\ref{algo:GPM} requires entry-wise bounds on key quantities in each iteration, which are more involved than $\ell_2$ bounds. In the next two lemmas, we will see which entry-wise bounds we need in order to establish the local contraction property.

Recall that $\mathcal{P}:\mathbb{C}^n \to \mathbb{C}_1^n$ maps each entry of a vector to the unit circle in the complex plane: 
\begin{equation*}
\forall\, k \in [n], \qquad ( \mathcal{P}(x) )_k = 
\begin{cases}
x_k / |x_k| & \textrm{ if } x_k \neq 0, \\ 1&  \textrm{ if } x_k = 0.
\end{cases}
\end{equation*}
The case $x_k = 0$ will not appear in the proofs because it can be excluded with high probability. Henceforth, we also drop parentheses in $\mathcal{P}(x)$ for simplicity.

In the next two lemmas, we establish the local contraction property of $\mathcal{P} \mathcal{L}$ (the GPM operator) under the distance $d_2$. Under certain conditions on the input points, $\mathcal{P} \mathcal{L}$ shrinks the $d_2$ distance between points by a ratio in $(0,1)$. In a rigorous sense, this does not imply that $\mathcal{P}\mathcal{L}$ is a contraction mapping, because the output points do not necessarily satisfy the conditions themselves. However, for the sequences of interest, the conditions are satisfied with high probability and this is all we need to ascertain convergence---see Theorems~\ref{thm::indct} and~\ref{thm::finCnvg}.

\begin{lemma}\label{lem::P0}
Suppose $\varepsilon \in [0,1)$. For any $x,y \in \mathbb{C}$, if $|x| \ge 1 - \varepsilon, |y| \ge 1 - \varepsilon$, then
\begin{equation}\label{ineqn::lemP0}
\left| \frac{x}{|x|} - \frac{y}{|y|} \right| \le (1 - \varepsilon)^{-1} | x - y |.
\end{equation}
As a consequence, for any $w,v \in \mathbb{C}^n$ with $\min_k|w_k| \ge 1- \varepsilon$ and $\min_k|v_k| \ge 1- \varepsilon$, 
\begin{equation}\label{ineqn::calPLip}
d_2( \mathcal{P} w , \mathcal{P}v ) \le (1 - \varepsilon)^{-1} d_2(w , v ).
\end{equation}
\end{lemma}

This lemma says $\mathcal{P}$ is Lipschitz continuous (in the quotient space $\mathbb{C}^n / \! \sim$) in a region where $|w_k|$ and $|z_k|$ are uniformly lower bounded. The composition $\mathcal{P} \mathcal{L}$, will have the local contraction property as long as the contraction ratio in (\ref{ineqn::ctr}) is small enough, and the $\varepsilon$ in (\ref{ineqn::calPLip}) is not too close to $1$. 
The next lemma formalizes this result. For later use in the proofs, we also introduce an additional notation: for a Hermitian matrix $W' \in \mathbb{C}^{n \times n}$, we let $\mathcal{L}' = (z z^* + \sigma W')/n$.

\begin{lemma}\label{lem::ctrT}
Suppose $x,y \in \sqrt{n}\, \mathbb{S}^{n-1}$, $K_1,K_2>0$ and $\varepsilon_1,\varepsilon_2 \in (0,1/2)$ with
\begin{align}
& d_2( x , z ) \le \varepsilon_1 \sqrt{n}, & &d_2( y , z ) \le \varepsilon_2 \sqrt{n}, \label{ineqn::ctr1}\\
& \| W x \|_{\infty} \le K_1 \sqrt{n \log n}, & & \| W' y \|_{\infty} \le K_2 \sqrt{n \log n}. \label{ineqn::ctr2}
\end{align}
Let $\varepsilon = \max\{ \varepsilon_1^2/2 + K_1 \sigma \sqrt{\log n / n},  \varepsilon_2^2/2 + K_2 \sigma \sqrt{\log n / n}\}$. If $\varepsilon < 1$, then
\begin{equation}\label{ineqn::ctrT1}
d_2( \mathcal{P} \mathcal{L} x , \mathcal{P} \mathcal{L}' y  ) \le ( 1 - \varepsilon)^{-1} d_2 (\mathcal{L}x , \mathcal{L}'y ).
\end{equation}
In particular, when $\mathcal{L} = \mathcal{L}'$, if $\varepsilon < 1$ we have 
\begin{equation}\label{ineqn::ctrT2}
d_2( \mathcal{P} \mathcal{L} x , \mathcal{P} \mathcal{L} y  ) \le \rho \cdot d_2( x , y ),
\end{equation}
where $\rho = (1 - \varepsilon)^{-1}( 6 \max\{ \varepsilon_1, \varepsilon_2\} + \sigma \| W \|_2 / n )$.
\end{lemma}
This deterministic lemma states that $\mathcal{T} := \mathcal{P} \mathcal{L}$ has the local contraction property in a region where (\ref{ineqn::ctr1}) and (\ref{ineqn::ctr2}) are satisfied, as long as $\rho < 1$. Note that we have to require $\ell_\infty$ bounds in~\eqref{ineqn::ctr2} because of the entry-wise nature of $\mathcal{P}$. 
In the next subsection, we use this lemma to show that, with high probability, $d_2(\calT x^{t}, \calT x^{t,m})$ is controlled by $\rho \cdot d_2(x^{t}, x^{t,m})$ where the ratio $\rho$ lies in $(0,1)$.

\subsection*{Convergence analysis}

Let us denote the rescaled matrix $C^{(m)}/n$ by $\mathcal{L}^{(m)}$, where $m\in [n]$. Also let $\mathcal{T}^{(m)} = \mathcal{P} \mathcal{L}^{(m)}$. Recall that $x^0$ is identical to the leading eigenvector $\tilde{x}$ of $\mathcal{L}$, and $x^{0,m}$ is identical to the leading eigenvector $\tilde{x}^{(m)}$ of $\mathcal{L}^{(m)}$; and they are normalized such that $\| x^0 \|_2 = \| x^{0,m} \|_2 = \sqrt{n}$. Algorithm~\ref{algo:GPM} iterates $x^{t+1} = \mathcal{T} x^t$. The auxiliary sequences defined for theoretical analysis (not implemented in practice) follow a similar update rule: $x^{t+1,m} = \mathcal{T}^{(m)} x^{t,m}$ (Figure~\ref{fig:contraction}). Note that, for all $t \ge 1$, each entry of $x^t$ and $x^{t,m}$ has unit modulus, i.e., $x^t, x^{t,m} \in \mathbb{C}_1^n $, but $x^0, x^{0,m}$ are not in $\mathbb{C}_1^n$ in general.

As shown in Lemma \ref{lem::ctrT}, the local contraction property of $\mathcal{T}$ hinges on the condition that the vectors to be updated are in the contraction region $\mathcal{N} = \mathcal{N}_1 \bigcap  \mathcal{N}_2$, where $\mathcal{N}_1$ and $\mathcal{N}_2$ are defined in (\ref{eq:None}) and (\ref{eq:Ntwo}). 
The absolute constants $\kappa_2, \kappa_3>0$ in their definitions will be specified in Theorem \ref{thm::indct}.

In order to show the iterates $x^t$ stay in $\mathcal{N}$, we analyze the dependence between the random quantities $\mathcal{T}$ and $x^t$ by making use of the auxiliary sequences $x^{t,m}$, as illustrated by eq.~(\ref{eq:inftyboundsplit2}). As shown in the  analysis of eigenvectors in Section~\ref{sec:eigen}, we know $d_2(x^0, x^{0,m})$ is small. Owing to the local contraction property (Lemma~\ref{lem::ctrT}), we can prove the recursive error bound~(\ref{recurProx}), which ensures that $d_2(x^{t,m}, x^t)$ is bounded throughout all iterations. The analysis is based on induction.

As discussed in Section \ref{sec:intro}, a technical issue is that we cannot use concentration results infinitely many times for all $t=1,2,\ldots$, because we use the union bound to achieve a high probability result. We study first $T+1 := 3n^2$ iterates using concentration results, and resort to deterministic analysis for later iterations. In the following theorems, constants $C_1', C_2'$ are those constants in Lemmas~\ref{lem::conctr2} and~\ref{lem::conctr1}.

\begin{theorem}\label{thm::indct}
Suppose $n \ge 2$ and $\sigma$ satisfies
\begin{equation}\label{ineqn::sigmaBound}
\sigma \le \min \left\{\frac{\sqrt{n}}{120\sqrt{2}C_2'}, \frac{1}{240\kappa_2}\sqrt{\frac{n}{\log n}} \right\}.
\end{equation}
Then, with probability $1 - \mathcal{O}(n^{-2})$, for any $t$ in $\{ 0, 1, 2, \ldots, T\}$,
\begin{align}
& d_2( x^{t,m} , x^t ) \le \kappa_1,  \quad \forall \, m \in [n], \label{ineqn::indct1} \\
& \| W x^t \|_{\infty} \le \kappa_2 \sqrt{n \log n},  \label{ineqn::indct2}\\
& d_2( x^t , z ) \le  \kappa_3\sqrt{n}. \label{ineqn::indct3}
\end{align}
Here $\kappa_1,\kappa_2, \kappa_3$ are absolute constants: $\kappa_1 = \kappa_3 = 1/60$ and  $\kappa_2 =  4C_1' + 2C_2' \kappa_1$.
\end{theorem} 

Note that (\ref{ineqn::indct2}) and (\ref{ineqn::indct3}) guarantee $\{x^t \}_{t=0}^T \subset \mathcal{N}$.
Considering~\eqref{recurProx} from our proof map,
\begin{equation*}
d_2( x^{t+1,m} , x^{t+1} ) \le d_2( \mathcal{T} x^{t,m} , \mathcal{T} x^t ) + d_2( \mathcal{T}^{(m)} x^{t,m} ,  \mathcal{T} x^{t,m}).
\end{equation*}

Assuming~(\ref{ineqn::indct2}) and (\ref{ineqn::indct3}) for the case $t$, by the local contraction property (Lemma~\ref{lem::ctrT}), the first term is bounded by $\rho \cdot d_2( x^{t,m} ,  x^t )$ where $\rho < 1$. By the concentration bounds (Lemma~\ref{lem::conctr1}), the second term is bounded by $ \mathcal{O} ( \sigma \| \Delta W^{(m)} x^{t,m} \|_2 / n) =  \mathcal{O}(\sigma \sqrt{\log n / n}) = \mathcal{O}(1)$ with high probability. 
Therefore, it is expected that (\ref{ineqn::indct1}) continues to hold for the case $t+1$.

This proximity property (\ref{ineqn::indct1}) is crucial to show that $x^{t+1}$ stays in the local region $\mathcal{N}$.
To bound $\| W x^{t+1} \|_{\infty}$, we use the concentration bounds (Lemma~\ref{lem::conctr1}) and the proximity property~(\ref{ineqn::indct1}) in the inequality~(\ref{eq:inftyboundsplit2}). 

To bound $d_2(x^{t+1}, z)$, we derive an entry-wise bound on $\mathcal{L}x^t$, then use Lemma~\ref{lem::P0}. This is straightforward once we have a bound on $\| W x^t \|_{\infty}$.

The next result says $d_2( x^{t+1} , x^t )$ decreases geometrically for $t = 0, \ldots, T-1$, which is notably useful to analyze later iterations $(t \geq T)$.
\begin{theorem}\label{thm::geoDcr}
Under the same assumption as in Theorem \ref{thm::indct}, with probability $1 - \mathcal{O}(n^{-2})$, we have
\begin{equation}\label{ineqn::cauchy}
d_2( x^{t+1} ,  x^{t} ) \le \frac{1}{2}\, d_2( x^t , x^{t-1} ), \qquad \forall t \in \{1, 2, \ldots, T-1\},
\end{equation}
and as a consequence, $d_2( x^{T} ,  x^{T-1}) \le 2^{2-T} \sqrt{n}$.
\end{theorem}

This result is similar to the contraction mapping theorem (though we cannot prove $\mathcal{T}x \in \mathcal{N}$ for all $x \in \mathcal{N}$), which says the sequence produced by a contraction mapping is a Cauchy sequence and satisfies an inequality similar to (\ref{ineqn::cauchy}). After $T = 3n^2 - 1$ iterations, the update from $x^t$ to $x^{t+1}$ is almost negligible. Although we no longer rely on concentration results, we will show, by induction again, that $d_2(x^t, x^{T})$ remains very small for all $t \ge T$. This ensures that $x^t$ stays in a slightly larger contraction region for all $t$ (with larger constants). The next theorem depends crucially on the conclusions of Theorem \ref{thm::indct} and \ref{thm::geoDcr}.

\begin{theorem}\label{thm::finCnvg}
Suppose $\|W\|_2 \le C_2' \sqrt{n}$, $\| W x^{T-1} \|_{\infty} \le \kappa_2 \sqrt{n \log n}$, $ d_2( x^{T-1} , z ) \le \kappa_3 \sqrt{n}$ and $\sigma$ is bounded
as in Theorem \ref{thm::indct}.
Also suppose $ d_2( x^T , x^{T-1} ) \le \kappa_3/4$. Then,
\begin{equation}\label{ineqn::y}
d_2( x^{T+k} , x^{T+k-1} ) \le 2^{-k} d_2( x^T , x^{T-1} ), \qquad\, \forall \, k \ge 0.
\end{equation}
Furthermore, in the quotient space $\mathbb{C}_1^n / \! \sim$ equipped with distance $d_2$, the sequence $[x^t]$ $(t \ge 1)$ converges to a limit $[x^\infty]$ where $x^\infty \in \mathbb{C}_1^n$ is a fixed point of $\mathcal{T}$, i.e.,  $\mathcal{T} x^{\infty} = x^{\infty}$. This fixed point satisfies
\begin{align}
d_2( x^{\infty} , z ) & \le \frac{3}{2}\kappa_3\sqrt{n},  \textrm{ and } & \|Wx^{\infty} \|_{\infty} & \le (\kappa_2 + C_2' \kappa_3) \sqrt{n \log n}.
\label{ineqn::Wyinf}
\end{align}
Moreover,
\begin{align}
C x^{\infty} & = \diag(\mu) x^{\infty},
\label{eqn::eigDual}
\end{align}
where $\mu_k = | (C x^{\infty})_k |$.
\end{theorem}

Under the stated conditions, this theorem is a deterministic result. By Theorems~\ref{thm::indct} and~\ref{thm::geoDcr}, the conditions hold with high probability (note that $2^{2-T} \sqrt{n} \le \kappa_3/4$ when $n\ge 2$).
This theorem establishes the convergence of $[x^t]$ and, importantly, the bounds~\eqref{ineqn::Wyinf} extend to the limit $[x^\infty]$ by continuity. This strong characterization of the limit point puts us in a favorable position to verify optimality.

\subsection*{Verifying optimality}
To verify optimality of $x^\infty$, it is convenient to use a known dual certificate for the SDP relaxation. The following is a combination of Lemmas~4.3 and~4.4 in~\citep{bandeira2014tightness}.
\begin{lemma}\label{lem:certificate}
	A feasible $X$ for~\eqref{eq:SDP} is optimal if and only if
	\begin{align}
		S(X) & := \Re\{\ddiag(CX)\} - C
		\label{eq:S}
	\end{align}
	is positive semidefinite, where $\ddiag$ sets all off-diagonal entries to zero. If furthermore $\rank(S) = n-1$, then $X$ is the unique solution of~\eqref{eq:SDP}, it is of the form $X = \hat{x} \hat{x}^*$ and $\hat{x}$ is the unique global optimum of~\eqref{eq:P} up to global phase.
\end{lemma}

To simplify notation, let $x = x^\infty$. Using the same developments as in~\citep[\S4.4]{bandeira2014tightness}, we verify that $S = S(xx^*)$ is positive semidefinite and has rank $n-1$ under condition~\eqref{ineqn::sigmaBound} on $\sigma$ and the conclusions of Theorem~\ref{thm::finCnvg}, namely, inequalities~\eqref{ineqn::Wyinf} and equation~\eqref{eqn::eigDual}. By construction, $Sx = 0$. Hence, it is sufficient to verify that $u^*Su > 0$ for all $u \in \Cn$ with $\|u\|_2 = 1$ and $u^*x = 0$:
\begin{align*}
	u^*Su & = \sum_{k \in [n]} |u_k|^2 \Re\{ (Cx)_k\overline{x_k} \} - u^*Cu \\
	& \stackrel{\eqref{eqn::eigDual}}{=} \sum_{k \in [n]} |u_k|^2 |(Cx)_k| - |u^*z|^2 - \sigma u^* W u \\
	& \geq |z^* x| - \sigma \|Wx\|_\infty - d_2^2(z, x) - \sigma \|W\|_2.
\end{align*}
(Owing to $u^*x = 0$, we used $|u^*z| = |u^*(z-xe^{i\theta})| \leq \|z-xe^{i\theta}\|_2 = d_2(z, x)$, with appropriate choice of $\theta$.) Now using $|z^*x| = n - \frac{1}{2}d_2^2(z, x)$ and assuming $\|W\|_2 \leq C_2'\sqrt{n}$ and the bounds~\eqref{ineqn::Wyinf} hold, it follows that
\begin{align*}
	u^*Su & \geq n - \frac{3}{2}d_2^2(z, x) - \sigma \|Wx\|_\infty - \sigma \|W\|_2 \\
	& \geq n - \frac{27}{8} \kappa_3^2 n - \sigma (\kappa_2 + C_2' \kappa_3) \sqrt{n \log n} - \sigma C_2'\sqrt{n}.
\end{align*}

Assume $\sigma$ satisfies inequality~\eqref{ineqn::sigmaBound}. Then, using $\kappa_1 = \kappa_3 = 1/60$ and $\kappa_2 = 4C_1' + 2C_2' \kappa_1 \ge 2C_2' \kappa_3$ as in Theorem~\ref{thm::indct},
\begin{align*}
	u^*Su \geq n \left( 1 - \frac{27}{8} \kappa_3^2 - \frac{\kappa_2 + \kappa_2/2 }{240\kappa_2} - \frac{1}{120\sqrt{2}}\right) > 0 
\end{align*}

Thus, $S$ is positive semidefinite and has rank $n-1$---which implies $xx^*$ is the unique solution of~\eqref{eq:SDP} and $x$ is the unique solution of~\eqref{eq:P} up to phase by Lemma~\ref{lem:certificate}---provided $\sigma$ satisfies~\eqref{ineqn::sigmaBound}, $\|W\|_2 \leq C_2'\sqrt{n}$ and the conclusions of Theorem~\ref{thm::finCnvg} hold. Theorem~\ref{thm:mainSDP2} follows directly; details for Theorems~\ref{thm:mainGPM2}--\ref{thm:mainbound2} are in the appendix.

\section{Conclusions and perspectives}

We proved that both semidefinite relaxation and the generalized power method are able to find the global optimum of (\ref{eq:P}) under the regime $\sigma = \mathcal{O}(\sqrt{n / \log n})$ with high probability. In other words, the maximum-likelihood estimator of phase synchronization is computationally feasible under noise level $\sigma = \mathcal{O}(\sqrt{n / \log n})$, which (nearly) matches the information-theoretic threshold, thus closing the gap in previous papers. We also derived $\ell_2$ and $\ell_\infty$ bounds on the optimum $\hat{x}$, and the $\ell_\infty$ bound improves upon previous results. The proof is based on tracking $n$ auxiliary sequences, which is a novel technique developed in this paper. As a by-product, we also proved an $\ell_\infty$ bound for the eigenvector estimator, which is of independent interest.

An interesting problem for future work is to prove (or disprove) that second-order necessary optimality conditions are sufficient for~(\ref{eq:P}). If this is true, then any algorithm that finds a second-order critical point also solves the nonconvex problem~(\ref{eq:P}). This was proved in~\cite{boumal2016nonconvexphase} for $\sigma  = \mathcal{O}(n^{1/6})$ then in~\cite{liu2016statistical} for $\sigma  = \mathcal{O}(n^{1/4})$. Numerical experiments in~\cite{boumal2016nonconvexphase} suggest that a local optimization method (namely, the Riemannian trust-region method) with random initialization finds the global optimum with $\sigma  = \tilde{\mathcal{O}}(n^{1/2})$ and random initialization. The analysis presented here does not apply directly though, because it hinges on a characterization of the limit points of GPM: a priori, this does not allow to characterize all second-order critical points.

A natural extension of our work is to establish similar results for synchronization over SO($d$) \citep{boumal2013MLE, wang2012LUD} and SE($d$) \citep{ozyesil2016synchronization}. The general synchronization problem is to recover group element $g_1,\ldots, g_n \in G$ from their noisy pairwise measurements $g_k^{-1} g_\ell^{}$. Our work here addresses synchronization over the group SO(2) (equivalently, the group U(1)), that is, in-plane rotations (equivalently, points on the unit complex circle). Another important group in practice is the rotation group SO(3), which is often used to describe the orientation of an object \cite{martinec2007robust, cucuringu2011eigenvector, shkolnisky2012viewing}. It is shown empirically in~\cite{boumal2015staircase} that the Riemannian trust-region method performs well. The analysis may be complicated by the fact that SO(3) is a non-commutative group.

Another important problem in practice is to handle incomplete measurement sets. In this paper, we suppose all entries of $C$ are known. A more realistic setting is that some pairs of phase differences are measured, forming edges of a graph. This appears in many applications \citep{stella2009angular, giridhar2006distributed}, and is addressed in a number of papers \citep{singer2010angular, cucuringu2012sensor}. The effect of an incomplete measurement graph on fundamental bounds is well understood as being related to the Laplacian of the graph \citep{crbsubquot}. See \cite{rosen2016certifiably} for robotics applications.

Finally, another problem of practical concern is robustness of estimation methods. Here,~(\ref{eq:P}) minimizes the sum of squared errors. However, in practice, more robust methods may be required to deal with outliers. In~\cite{wang2012LUD} for example, the authors minimize a sum of unsquared errors. A common way to solve such problems is via iterative reweighted least squares (IRLS), which is widely used in statistics~\citep{holland1977robust}. IRLS solves a weighed least squares problem in each step, where the weights depend on the current iterate. In this regard, our analysis could be a first step toward understanding robust methods with IRLS for synchronization problems.

\section*{Acknowledgements}
The authors thank Afonso Bandeira, Amelia Perry, Alexander Wein, Amit Singer and Jianqing Fan for helpful discussions.

\appendix
\section{Proofs}

\subsection*{Proofs for Section \ref{sec:eigen}}

\begin{proof}[Proof of Lemma \ref{lem::conctr2}]
Let $M_u$ be the upper triangular part of $W$, i.e., $(M_u)_{ij} = W_{ij} \mathbf{1}_{i \le j}$, and $M_l = W - M_u$. Then $\Re(M_u), \Im(M_u), \Re(M_l), \Im(M_l)$ are all matrices with independent and sub-gaussian entries, whose sub-gaussian moments are bounded by an absolute constant (see \cite{Ver10} for equivalent definitions of sub-gaussian variables). We can then apply Proposition 2.4 in \cite{RudVer10} and obtain the desired concentration bound on $W$. For $W^{(m)}$, we take a union bound over choice of $m$. The bound on $\Delta W^{(m)}$ follows from those on $W$ and $W^{(m)}$. The bound on $\| w_m \|_2$ follows from that on $\| W \|_2$, since
\begin{align*}
\| W \|_2  & = \sup_{\| u \|_2 = 1} \| W u \|_2 \ge \sup_{\| u \|_2 = 1} | w_m^* u | = \| w_m \|_2.
\end{align*}
\end{proof}

\begin{proof}[Proof of Lemma \ref{lem::conctr1}]
We will prove this lemma in the case where $\mathcal{U}^{(m)}$ is a \textit{deterministic} set. The case where $\mathcal{U}^{(m)}$ is random follows easily from the deterministic case, since we can first condition on $\mathcal{U}^{(m)}$ and use independence between $\Delta W^{(m)}$ and $\mathcal{U}^{(m)}$. In the proof, notations $C_1, C_1', C_2, c_1, c_2 > 0$ denote some absolute constants. For a fixed $m \in [n]$ and $u \in \mathcal{U}^{(m)}$, 
\begin{equation}\label{eqn::DeltaW}
\| \Delta W^{(m)} u \|_2^2 = |(\Delta W^{(m)} u)_m|^2 + \sum_{k \neq m} | (\Delta W^{(m)} u)_k |^2.
\end{equation}
We will bound the two parts on the right-hand side separately. We can expand $(\Delta W^{(m)} u)_m$ into a sum: $(\Delta W^{(m)} u)_m = \sum_k W_{mk} u_k$. By assumption, $\Re(W_{mk})$ is a sub-gaussian random variable. By Hoeffding's inequality for sub-gaussian variables \citep{Ver10}, 
\begin{equation*}
\mathbb{P}\left( \left|\sum_{k=1}^n \Re(W_{mk}) \Re(u_k) \right| \ge t\right) \le \exp\left( 1 - \frac{c_1 t^2}{n} \right).
\end{equation*}
For sums of random variables of $\Re(W_{mk}) \Im(u_k)$, $\Im(W_{mk}) \Re(u_k)$ or $\Im(W_{mk}) \Im(u_k)$ over $k$, similar concentration results hold. Thus, we can set $t = C_1 \sqrt{n \log n }$, where $C_1>0$ is some large absolute constant such that $c_1 C_1^2 \ge 5$, and deduce that with probability at least $1 - 4e \cdot n^{-5}$, 
\begin{equation*} 
|(\Delta W^{(m)} u)_m| \le 4C_1\sqrt{n \log n }.
\end{equation*}
Now let us bound the second term in the right-hand side of (\ref{eqn::DeltaW}). Observe 
\begin{equation}\label{ineqn::tmp}
\sum_{k \neq m} | (\Delta W^{(m)} u)_k |^2 = \sum_{k \neq m} | W_{km} u_m |^2 = \sum_{k \neq m} \left( \Re(W_{km})^2 + \Im(W_{km})^2 \right) |u_m|^2 .
\end{equation}
Since $\Re(W_{km})$ is sub-gaussian, it follows that $\Re(W_{km})^2$ is sub-exponential with a bounded sub-exponential norm. So we can use Bernstein's inequality for sub-exponential random variables \citep{Ver10}, 
\begin{equation*}
\mathbb{P}\left( \left|\sum_{k=1}^n ( \Re(W_{km})^2 - \mathbb{E} [\Re(W_{km})]^2 ) \right| \ge t \right) \le 2 \cdot \exp\left( - c_2 \min\big\{ \frac{t^2}{4n},  \frac{t}{2} \big\} \right).
\end{equation*}
A similar concentration bound holds for $\sum_{k} \Im(W_{km})^2$. Setting $t =n$, we know that with probability at least $1 - 4 e^{-c_2n/4}$, 
\begin{equation*}
\left| \sum_{ k= 1}^n (|W_{km}|^2 - \mathbb{E}|W_{km}|^2) \right| \le 2n.
\end{equation*}
From an equivalent definition of sub-gaussian variables (see \cite{Ver10}),  we obtain $\sum_{k = 1}^n  \mathbb{E}|W_{km}|^2 \le C_2n$, so it follows that $\sum_{k \neq m} | (\Delta W^{(m)} u)_k |^2 \le (2+C_2)n M_m^2 $, where $M_m$ is a bound on $|u_m|$, uniform over $\mathcal{U}^{(m)}$. Therefore, combining the upper bounds for the two terms in (\ref{eqn::DeltaW}), we deduce that for some large absolute constant $C_1'> 0$, 
\begin{equation*}
\| \Delta W^{(m)}u \|_2\le \big( 16C_1^2 n \log n  +  (2+C_2)n  M_m^2 \big)^{1/2} \le C_1' \sqrt{n \log n} + C_1' \sqrt{n} \, M_m
\end{equation*}
holds with probability $1 -  4en^{-5} - 4e^{-c_2n/4}$. Taking a union bound over the choice of $u \in \mathcal{U}^{(m)}$ and $m$, we conclude that (\ref{ineqn::conctr1}) holds with probability $1 -  12en^{-2} - 12n^3e^{-c_2n/4}$, or equivalently $1 -  \mathcal{O}(n^{-2})$.
\end{proof}

\begin{proof}[Proof of Theorem \ref{thm::mainEig2}]
In this proof, we use $a(n) \lesssim b(n)$ to mean there exists an absolute constant $C$ such that $a(n) \le Cb(n)$. We also suppose $\sigma < c_0' \sqrt{n}$ where $c_0' < (8C_2')^{-1}$, and $C_2'$ is the absolute constant in Lemma \ref{lem::conctr2}. With probability $1 - \mathcal{O}(n^{-2})$, Lemma \ref{lem::conctr2} and Lemma \ref{lem::conctr1} hold, so we can safely use the concentration bounds. First we note the second part of Theorem \ref{thm::mainEig2} (about $\mathcal{P} \tilde{x} $) follows directly from \cite[Prop.~1]{liu2016statistical} once the first part is proved.

It is also straightforward to bound $d_2( \tilde{x} , z )$. Since $zz^*$ is a rank-$1$ matrix, the eigengap $\delta(zz^*) = \lambda_1(zz^*) - \lambda_2(zz^*)$ is simply its leading eigenvalue, which is $n$, From Lemma \ref{lem::ptb}, clearly 
\begin{equation*}
d_2( \tilde{x} , z ) \le \frac{ \sqrt{2} \sigma \| W z \|_2}{ n - \sigma \| W \|_2} \le \frac{ \sqrt{2}\,\sigma C_2' n}{n - C_2' \sigma \sqrt{n}} \le \frac{8\sqrt{2}}{7} C_2' \sigma \lesssim \sigma.
\end{equation*}
This leads to the first claim of the theorem. 

Now let us consider bounding $\| W \tilde{x} \|_{\infty}$. Following the inequality (\ref{ineqn::eigKey}), we only need to bound the two parts separately in (\ref{ineqn::eigKey}). By Weyl's inequality, $\delta(C^{(m)}) = \lambda_1(C^{(m)}) - \lambda_2(C^{(m)}) \ge \lambda_1(zz^*) - 2 \sigma\| W^{(m)} \|_2$, so Lemma \ref{lem::ptb} implies
\begin{equation} \label{ineqn::d2xtilde}
d_2( \tilde{x} ,  \tilde{x}^{(m)}) \le \frac{\sqrt{2}\sigma \| \Delta W^{(m)} \tilde{x}^{(m)} \|_2}{\delta(C^{(m)}) - \sigma \| \Delta W^{(m)} \|_2} \le \frac{8\sqrt{2}\, c_0' }{5\sqrt{n}}  \|  \Delta W^{(m)} \tilde{x}^{(m)} \|_2 
\end{equation}
Therefore, from (\ref{ineqn::eigKey}) we have $\| W \tilde{x} \|_{\infty} \lesssim \max_{m}( | w_m ^* \tilde{x}^{(m)}| + \| \Delta W^{(m)} \tilde{x}^{(m)} \|_2)$. Note the first term within the parenthesis is dominated by the second, since $w_m ^* \tilde{x}^{(m)}$ is exactly the $m$th coordinate of $\Delta W^{(m)} \tilde{x}^{(m)}$. Lemma \ref{lem::conctr1} implies that with probability $1- \mathcal{O}(n^{-2})$, 
\begin{equation} \label{ineqn::Wxtilde}
\| W \tilde{x} \|_{\infty} \lesssim  \max_{1 \le m \le n} \| \Delta W^{(m)} \tilde{x}^{(m)} \|_2 \lesssim \sqrt{n \log n} + \sqrt{n} \max_{1 \le m \le n} | \tilde{x}^{(m)}_m |.
\end{equation}
We claim that $\max_m  | \tilde{x}^{(m)}_m | < 8/7 $. Indeed, for any $m \in [n]$, by definition of $\tilde{x}^{(m)}$,
\begin{equation*}
\lambda_1(C^{(m)}) \tilde{x}^{(m)} = C^{(m)}\tilde{x}^{(m)} = (z^* \tilde{x}^{(m)}) z + \sigma W^{(m)} \tilde{x}^{(m)}.
\end{equation*}
Since $W^{(m)}$ has zero entries in its $m$th row, the $m$th coordinate of $W^{(m)} \tilde{x}^{(m)}$ vanishes, and we deduce
\begin{equation}\label{ineqn::Mmbound}
|\tilde{x}^{(m)}_m| = \frac{  |(z^* \tilde{x}^{(m)}) z_m| }{\lambda_1(C^{(m)})} = \frac{ |z^* \tilde{x}^{(m)}| }{\lambda_1(C^{(m)})} \le \frac{ |z^* \tilde{x}^{(m)}| }{n - \sigma \| W^{(m)} \|_2} \le \frac{ n}{n - C_2' \sigma \sqrt{n}} < 8/7,
\end{equation}
where we used Weyl's inequality $\lambda_1(C^{(m)}) \ge \lambda_1(zz ^*) - \sigma \| W^{(m)} \|_2$. This leads to $\max_m  | \tilde{x}^{(m)}_m | < 8/7 $, and therefore $\| W \tilde{x} \|_{\infty} \lesssim \sqrt{n \log n}$. This $\ell_\infty$ bound is directly related to $d_\infty( \tilde{x}, z)$. We choose the global phase of $\tilde x$ such that $\tilde x^* z = |\tilde x^* z|$, and thus $d_2(\tilde x, z)^2 = 2(n - |\tilde x^* z|)$. For any $m \in [n]$, $(C\tilde{x})_m =  (z^* \tilde{x}) z_m + \sigma (W \tilde{x})_m$, so
\begin{equation*}
| \tilde{x}_m - z_m| = \left|\frac{(C \tilde{x})_m}{\lambda_1(C)}  - z_m\right| \le \left|\frac{ z^* \tilde{x} }{\lambda_1(C)} - 1 \right| + \frac{\sigma\| W \tilde{x} \|_{\infty}}{\lambda_1(C)}.
\end{equation*}
From $d_2(\tilde{x}, z ) \lesssim \sigma$ and the identity $d_2(\tilde x, z)^2 = 2(n - |\tilde x^* z|)$, we have $ n \ge | z^* \tilde{x}| \ge n - \mathcal{O}(\sigma^2)$. By Weyl's inequality $|\lambda_1(C) - n| \le \sigma \|W \|_2 \lesssim \sigma\sqrt{n}$, and therefore,
\begin{align*}
\| \tilde{x} - z \|_{\infty} & \lesssim \frac{\sigma^2 + \sigma \sqrt{n}}{n} + \frac{\sigma \sqrt{n\log n} }{n} \lesssim \sigma \sqrt{\log n / n}.
\end{align*}
\end{proof}

\subsection*{Proofs for Section \ref{sec:phaseSyn}}

\begin{proof}[Proof of Lemma \ref{lem::ctr}]
Let us decompose $x,y$ into two parts that are orthogonal:
\begin{equation}\label{eqn:expand}
x = az + \sqrt{n}\,\alpha, \qquad y = bz + \sqrt{n}\,\beta,
\end{equation}
where $a,b \in \mathbb{C}; \alpha, \beta \in \mathbb{C}^n$, and $\alpha^* z  = \beta^* z  = 0$. Without loss of generality, we assume $a,b$ are real and $a,b \ge 0$, since we can freely choose the global phases of $x$ and $y$. Also suppose we choose $\theta \in \mathbb{R}$ such that $\| e^{i \theta} x - y \|_2$ is minimized, i.e., $\| e^{i \theta} x - y \|_2 = d_2(x,y)$. The key part of the proof is to show:
\begin{align}\label{ineqn:fix}
\frac{1}{n} |  z^* (e^{i \theta} x - y) | = |e^{i \theta}a - b| \le 6 \varepsilon \| e^{i \theta} \alpha - \beta \|_2.
\end{align}
While the above equality is easily obtained by expressing $x,y$ according to (\ref{eqn:expand}), the difficulty lies in showing the inequality. Once proved, this immediately leads to the desired inequality (\ref{ineqn::ctr}), because (using $\|e^{i \theta} \alpha - \beta\|_2 \leq d_2(x, y) / \sqrt{n}$)
\begin{align*}
\| e^{i \theta}  \mathcal{L}x - \mathcal{L} y \|_2 &\le \frac{1}{\sqrt{n}} | z^* (\etheta x - y) | + \frac{\sigma}{n} \| W (\etheta x) - Wy \|_2 \\
&\le 6 \varepsilon \sqrt{n}\, \| e^{i \theta} \alpha - \beta \|_2 + \frac{\sigma \| W \|_2}{n} \| \etheta x - y \|_2 \\
&= \left(6 \varepsilon + \frac{\sigma \| W \|_2}{n} \right) d_2(x,y) .
\end{align*}
Since $d_2(\mathcal{L}x , \mathcal{L} y) \le \| e^{i \theta}  \mathcal{L}x - \mathcal{L} y \|_2$, the proof will be complete. Note that using Cauchy--Schwarz inequality is futile as it leads to $| z^* (e^{i \theta} x - y) | \le \sqrt{n}\,d_2(x,y)$, which cannot be used to show that $\mathcal{L}$ is Lipschitz with a constant $\rho < 1$. An important intuition is that, if $x$ and $y$ are close to $z$, and $e^{i\theta}$ is close to $1$, then $e^{i\theta}x-y$ cannot be aligned too much with $z$. The rest of the proof is devoted to showing (\ref{ineqn:fix}).

Since $d_2( x , z )^2 = 2n(1 - a), d_2( y , z )^2 = 2n(1 - b)$, the conditions $d_2( x , z ) \le \varepsilon \sqrt{n}$ and $d_2( y , z ) \le \varepsilon \sqrt{n}$ are equivalent to 
\begin{equation}\label{ineqn::aalpha}
a  \ge 1-  \varepsilon^2/2, \qquad b \ge 1 - \varepsilon^2/2.
\end{equation}
Since $\varepsilon < 1/2$, we must have $a + b > 2(1 - 1/8) > 1$, and thus $| a - b | \le (a+b) |a-b| = | a^2 - b^2 |$. 
Since $\| x \|_2 = \| y \|_2 = \sqrt{n}$, we know $a^2 + \| \alpha \|_2^2 = 1$ and $b^2 + \| \beta \|_2^2 = 1$, so we have $a^2 - b^2 =  \| \beta \|_2^2 - \| \alpha \|_2^2$. Using the inequalities $\| \alpha \|_2 \le d_2(x,z)/\sqrt{n} \le  \varepsilon, \| \beta \|_2 \le d_2(y,z)/\sqrt{n} \le \varepsilon$, we derive
\begin{align}\label{ineqn:aminusb}
|a - b|  \le ( \| \alpha \|_2 + \| \beta \|_2 ) \left| \| \alpha \|_2 - \| \beta \|_2 \right| \le 2\varepsilon \| e^{i \theta} \alpha - \beta \|_2,
\end{align}
where we used the triangular inequality in the second inequality. By the choice of $\theta$, we have $\| e^{i \theta} x - y \|_2^2 \le \| x - y \|_2^2$, or equivalently
\begin{align*}
| e^{i \theta}a - b |^2 + \| e^{i \theta} \alpha - \beta \|_2^2 \le |a - b|^2 + \| \alpha - \beta \|_2^2.
\end{align*}
Since $| e^{i \theta}a - b |^2 - |a - b|^2 = 2ab(1 - \cos \theta)$, and by the triangular inequality, $ \| \alpha - \beta \|_2 - \| e^{i \theta} \alpha - \beta \|_2 \le | 1 - e^{i \theta}|  \| \alpha \|_2 \le \varepsilon | 1 - e^{i \theta}|$, we obtain
\begin{align*}
2ab(1 - \cos \theta) &\le \| \alpha - \beta \|_2^2 - \| e^{i \theta} \alpha - \beta \|_2^2 \\
&\le \varepsilon | 1 - e^{i \theta}| \left( \| \alpha - \beta \|_2 + \| e^{i \theta} \alpha - \beta \|_2 \right) \\
&\le \varepsilon | 1 - e^{i \theta}|  \left( \varepsilon | 1 - e^{i \theta}| + 2\| e^{i \theta} \alpha - \beta \|_2\right).
\end{align*}
Notice that $2(1 - \cos \theta) =  | 1 - e^{i \theta}|^2$, and thus we derive
\begin{align*}
(\varepsilon^{-1} ab - \varepsilon)  | 1 - e^{i \theta}| \le 2\| e^{i \theta} \alpha - \beta \|_2.
\end{align*}
From (\ref{ineqn::aalpha}), we have $ab \ge (1 - \varepsilon^2/2)^2 > 1 - \varepsilon^2$, so
\begin{align*}
\varepsilon^{-1} ab - \varepsilon \ge \varepsilon^{-1}(1 - \varepsilon^2) - \varepsilon = \varepsilon^{-1}(1 - 2\varepsilon^2) > (2\varepsilon)^{-1},
\end{align*}
where the last inequality is due to the condition $\varepsilon < 1/2$. Therefore, we deduce $| 1 - e^{i \theta}| \le 4 \varepsilon \| e^{i \theta} \alpha - \beta \|_2$. Combining with (\ref{ineqn:aminusb}), we obtain (recall that $a \leq 1$)
\begin{align*}
|e^{i \theta}a - b| = |e^{i \theta}a - a + a - b| \le |e^{i \theta} - 1|a + |a - b| \le 6 \varepsilon \| e^{i \theta} \alpha - \beta \|_2 .
\end{align*}
This verifies (\ref{ineqn:fix}) and completes the proof. 
\end{proof}

\begin{proof}[Proof of Lemma \ref{lem::P0}]
By the cosine formula of triangles, 
\begin{equation*}
| x - y|^2 = |x|^2 + |y|^2 - 2|x||y| \cos \theta,
\end{equation*}
where $\theta = |\Arg(x) - \Arg(y)|$ is the angle formed by $x$ and $y$. Using the AM-GM inequality (i.e., $|x|^2 + |y|^2 \geq 2|x||y|$ which follows from $| |x|-|y| |^2 \geq 0$), we have
\begin{equation}\label{ineqn::cos}
|x-y|^2 \ge 2|x||y| -  2|x||y| \cos \theta \ge 2(1-\varepsilon)^2(1 - \cos \theta).
\end{equation}
Using the same cosine formula, we also have 
\begin{equation*}
\left| \frac{x}{|x|} - \frac{y}{|y|} \right|^2 = 2 - 2\cos \theta.
\end{equation*}
This equality, together with (\ref{ineqn::cos}), leads to the desired inequality (\ref{ineqn::lemP0}). Using (\ref{ineqn::lemP0}) for all $k \in [n]$, it follows that $\| \mathcal{P}w - \mathcal{P}v \|_2 \le (1 - \varepsilon)^{-1} \| w - v \|_2$. The same is true if we replace $w$ by $w e^{i \theta}$ with an arbitrary $\theta \in \mathbb{R}$. The second inequality (\ref{ineqn::calPLip}) is obtained by minimizing $\|we^{i \theta} - v\|_2$ over $\theta$.
\end{proof}

\begin{proof}[Proof of Lemma \ref{lem::ctrT}]
We will first show that the moduli of $(\mathcal{L}x)_k$ and $(\mathcal{L}'y)_k$ are uniformly lower bounded for all $k \in [n]$, so that we can use Lemma \ref{lem::P0}. Since $d_2(z, x)^2 = 2(n - |z^*x|)$ and $d_2(z, x)^2 \leq n \varepsilon_1^2$, it follows that $2|z^*x| \geq n(2-\varepsilon_1^2)$. Thus, for any $k \in [n]$,
\begin{equation*}
\left| ( \mathcal{L} x )_k \right| = \frac{1}{n} \left|  (z^*x) z_k + \sigma (Wx)_k  \right| \ge \frac{1}{n} \left|  z^* x  \right|  - \frac{\sigma}{n} \left| (Wx)_k \right|  \ge 1 - \frac{1}{2} \varepsilon_1^2 - K_1\sigma \sqrt{\frac{\log n}{n}}.
\end{equation*}
Similarly, we have $| (\mathcal{L}' y )_k | \ge 1 - \varepsilon_2^2/2 - K_2 \sigma \sqrt{\log n/n}$. Now that $| ( \mathcal{L} x )_k |$ and $| (\mathcal{L}' y )_k |$ are uniformly lower bounded by $1 - \varepsilon$, where $\varepsilon := \max\{  \varepsilon_1^2/2  + K_1\sigma \sqrt{\log n/n}, \varepsilon_2^2/2 + K_2 \sigma \sqrt{\log n/n} \}$ is assumed smaller than $1$, we are able to invoke Lemma \ref{lem::P0}, and deduce
\begin{equation*}
d_2 (\mathcal{P}  \mathcal{L} x , \mathcal{P} \mathcal{L}' y ) \le (1 - \varepsilon)^{-1} d_2( \mathcal{L} x ,  \mathcal{L}' y).
\end{equation*}
This proves the first claim of the lemma; and the second claim follows directly from Lemma \ref{lem::ctr}. 
\end{proof}

\begin{proof}[Proof of Theorem \ref{thm::indct}]
Denote $C_1'' = 4 C_1' $. We note that $x^{0,m}$ requires special treatment in the proof. Indeed, for $t \ge 1$, $x^{t,m}$ has unit-modulus entries, whereas $x^{0,m}$ does not (in general). From Lemma \ref{lem::conctr2}, we deduce that with probability $1 - \mathcal{O}(n^{-2})$, the matrix norm bounds (\ref{ineqn::W}) hold. In Lemma \ref{lem::conctr1}, we set $\mathcal{U}^{(m)} = \{x^{t,m} \in \mathbb{C}^n: t = 0,\ldots, T\}$. By construction, $x^{t,m}$ is a function of $W^{(m)}$, and is independent of $\Delta W^{(m)}$. Besides, we have $|\mathcal{U}^{(m)}| = T+1 \le 3n^2$. This means the conditions in Lemma \ref{lem::conctr1} are satisfied, so with probability $1 - \mathcal{O}(n^{-2})$,
\begin{equation}\label{ineqn:maxt}
\max_{0 \le t \le T } | w_m ^* x^{t,m} | \le \max_{0 \le t \le T } \| \Delta W^{(m)} x^{t,m} \|_2 \le C_1'' \sqrt{ n\log n }, \quad \forall \, m \in [n].
\end{equation}
This is because for $t \ge 1$, we can use (\ref{ineqn::conctr1new}) directly; and for $t=0$, we use (\ref{ineqn::conctr1}) and derive $\| \Delta W^{(m)} x^{0,m} \|_2 \le C_1'\sqrt{ n\log n } + 2C_1'\sqrt{n} \le C_1'' \sqrt{ n\log n }$ due to $\max_{m \in [n]} M_m \le 2$ (derived in (\ref{ineqn::Mmbound})).

\underline{First we verify (\ref{ineqn::indct1})--(\ref{ineqn::indct3}) for $t = 0$}. The initializers $x^0$ and $x^{0,m}$ are simply eigenvectors of $C$ and $C^{(m)}$, and their bounds have been studied in Section \ref{sec:eigen}. Recall that $x^{0,m}$ is independent of $\sigma \Delta W^{(m)}$, so applying Lemma \ref{lem::ptb} (Davis--Kahan) yields
\begin{equation*}
d_2(  x^{0,m} , x^0 ) \le \frac{\sqrt{2}\, \sigma \| \Delta W^{(m)}  x^{0,m} \|_2}{\delta(C^{(m)}) - \sigma \| \Delta W^{(m)} \|_2} \le \frac{\sqrt{2}\, C_1''\sigma\sqrt{n \log n}}{n - 3 \sigma \| \Delta W^{(m)} \|_2} \le \frac{\sqrt{2}\, \kappa_2 \sigma\sqrt{n \log n}}{n/2} < \kappa_1,
\end{equation*}
Here, in the second inequality, we used the bounds $ \| \Delta W^{(m)}  x^{0,m} \|_2 \le C_1'' \sqrt{n \log n}$ by~\eqref{ineqn:maxt}, $ \|\Delta W^{(m)} \|_2 \le C_2' \sqrt{n}$, and $\delta(C^{(m)}) = \lambda_1(C^{(m)}) - \lambda_2(C^{(m)}) \ge n - 2\sigma \|\Delta W^{(m)}\|_2$ by Weyl's inequality. In the third inequality, we used $C_1'' \le \kappa_2 $ and $3\sigma \| \Delta W^{(m)} \|_2 \le 3C_2'\sigma \sqrt{n} < n/2$. The final inequality is due to the condition (\ref{ineqn::sigmaBound}). This verifies (\ref{ineqn::indct1}) for the base case $t=0$.

This immediately leads to  (\ref{ineqn::indct2}), since by the reasoning in (\ref{ineqn::eigKey}), we derive
\begin{align*}
\| W x^0 \|_{\infty} &\le \max_{1 \le m \le n} \big[ | w_m ^* x^{0,m} | + \| w_m \|_2 \cdot d_2( x^{0,m} , x^0) \big] \\
&\le C_1''  \sqrt{n \log n} + C_2' \sqrt{n} \cdot d_2( x^{0,m} , x^0 ) \\
&\le (C_1'' + 2C_2' \kappa_1)  \sqrt{n \log n} \\
&= \kappa_2\sqrt{n \log n}.
\end{align*}
Finally, applying Lemma \ref{lem::ptb} (Davis--Kahan) again, we obtain (\ref{ineqn::indct3}):
\begin{equation*}
d_2( x^0 , z ) \le \frac{ \sqrt{2} \, \sigma \| W z\|_2}{n - \sigma \| W \|_2} \le \frac{\sqrt{2}\, \sigma \cdot C_2'  n  }{n - C_2' \sqrt{n}\,\sigma } \le  \frac{n^{3/2}/120}{n/2} =\kappa_3\sqrt{n}.
\end{equation*}

Now suppose (\ref{ineqn::indct1})--(\ref{ineqn::indct3}) are true for $t$, and our goal is to derive the same inequalities for $t + 1$, until $t$ reaches $T$.

(i) \underline{Verifying (\ref{ineqn::indct1}) for $t+1$}: notice that 
\begin{equation*}
d_2( x^{t+1,m} , x^{t+1} ) \le d_2( \mathcal{T}^{(m)} x^{t,m} ,  \mathcal{T} x^{t,m} ) + d_2( \mathcal{T} x^{t,m} , \mathcal{T} x^t ).
\end{equation*}
We shall use Lemma \ref{lem::ctrT} to bound the above two terms. Before doing so, let us examine the conditions required in Lemma \ref{lem::ctrT}. In order to bound the first term, notice that 
\begin{equation*}
d_2( x^{t,m} , z ) \le d_2( x^{t,m} , x^{t} ) + d_2( x^{t} ,  z ) \le \kappa_1 + \kappa_3 \sqrt{n} \le (\kappa_1 + \kappa_3)\sqrt{n}.
\end{equation*}
Furthermore, let us check
\begin{align*}
\| W x^{t,m} \|_{\infty} & \le \mtheta \| W(\etheta x^{t,m} - x^t) \|_{\infty} + \|Wx^t\|_{\infty} \\
&\le \| W \|_2 \cdot d_2( x^{t,m} , x^t ) + \|Wx^t\|_{\infty} \\
&\le C_2' \kappa_1\sqrt{n} + \kappa_2 \sqrt{n \log n},
\end{align*}
where we used the trivial inequality $\| \cdot \|_{\infty} \le \| \cdot \|_2$.  Moreover, using~\eqref{ineqn:maxt},
\begin{align*}
\| W^{(m)} x^{t,m} \|_{\infty} &\le \| \Delta W^{(m)} x^{t,m} \|_{\infty} + \| W x^{t,m} \|_{\infty} \\
& \le  \| \Delta W^{(m)} x^{t,m} \|_2 + \| W x^{t,m} \|_{\infty} \\
& \le C_1'' \sqrt{n \log n} + C_2' \kappa_1\sqrt{n} + \kappa_2 \sqrt{n \log n} \\
& \le (C_1'' + 2C_2' \kappa_1 + \kappa_2) \sqrt{n \log n} = 2\kappa_2\sqrt{n \log n}.
\end{align*}
where we used the definition $\kappa_2 =  C_1'' + 2C_2' \kappa_1$. Applying Lemma \ref{lem::ctrT}, we obtain, by~\eqref{ineqn:maxt}
\begin{align*}
d_2( \mathcal{T}^{(m)} x^{t,m} , \mathcal{T} x^{t,m} ) &\le ( 1 - \eta_1)^{-1} d_2( \mathcal{L}^{(m)} x^{t,m} , \mathcal{L} x^{t,m} ) \\
&\le (1 - \eta_1)^{-1} \| ( \mathcal{L}^{(m)} - \mathcal{L} ) x^{t,m} \|_2 \\
&= (1 - \eta_1)^{-1} \sigma \| \Delta W^{(m)} x^{t,m} \|_2 / n \\
&\le (1 - \eta_1)^{-1} \sigma C_1'' \sqrt{\log n /n},
\end{align*}
where by condition (\ref{ineqn::sigmaBound}), 
\begin{equation*}
\eta_1 := (\kappa_1 + \kappa_3)^2/2 + 2\kappa_2 \sigma \sqrt{\log n/n} \le  \frac{1}{1800} +  \frac{1}{120} < 1/2.
\end{equation*}
This now leads to 
\begin{equation}\label{ineqn::theta1}
d_2 (\mathcal{T}^{(m)} x^{t,m} ,  \mathcal{T} x^{t,m})  \le 2\sigma C_1'' \sqrt{\log n /n} \le 2\sigma \kappa_2 \sqrt{\log n /n} \le \frac{1}{120} = \frac{\kappa_1}{2},
\end{equation}
by (\ref{ineqn::sigmaBound}) again. We can bound $d_2 ( \mathcal{T} x^{t,m} , \mathcal{T} x^t )$ similarly. From the inequalities
\begin{align*}
& d_2( x^t , z ) \le \kappa_3 \sqrt{n}, & & d_2( x^{t,m} , z ) \le \kappa_3 \sqrt{n} + \kappa_1 \le (\kappa_1 + \kappa_3)\sqrt{n}, \\
& \| W x^t \|_{\infty} \le \kappa_2 \sqrt{n \log n}, & & \| W x^{t,m} \|_{\infty} \le C_2' \kappa_1\sqrt{n} + \kappa_2 \sqrt{n \log n},
\end{align*}
and Lemma \ref{lem::ctrT}, we derive
\begin{align*}
d_2( \mathcal{T} x^{t,m} , \mathcal{T} x^t ) &\le (1-\eta_2)^{-1}(6(\kappa_1+\kappa_3) + \sigma \|W \|_2 / n) \cdot d_2( x^{t,m} , x^t ) \\
&\le 2 \cdot (0.2 + \sigma \|W \|_2 / n) \kappa_1,
\end{align*}
where, since $2C_2' \kappa_1 < C_1'' + 2C_2' \kappa_1 = \kappa_2 $,
\begin{equation*}
\eta_2 := (\kappa_1 + \kappa_3)^2/2 + (2C_2' \kappa_1 + \kappa_2)\sigma \sqrt{\log n/n} < \eta_1 \le 1/2.
\end{equation*}
Since $0.2 + \sigma \|W \|_2 / n \le 0.2 + \sigma C_2 / \sqrt{n} <1/4$ by condition (\ref{ineqn::sigmaBound}), it follows that 
\begin{equation}\label{ineqn::theta2}
d_2( \mathcal{T} x^{t,m} , \mathcal{T} x^t ) \le \kappa_1/2.
\end{equation}
Therefore, bounds (\ref{ineqn::theta1}) and (\ref{ineqn::theta2}) imply that $ d_2( x^{t+1,m} , x^{t+1} ) \le \kappa_1$. 

(ii) \underline{Verifying (\ref{ineqn::indct2}) for $t+1$}: 
Following (\ref{eq:inftyboundsplit2}), for any $m \in [n]$, we obtain
\begin{align*}
| (Wx^{t+1})_m | &\le | w_m ^* x^{t+1,m} | +  \| w_m \|_2 \cdot d_2( x^{t+1} ,  x^{t+1,m} )  \\
&\le C_1''\sqrt{n \log n} + C_2'\sqrt{n}\cdot \kappa_1,
\end{align*}
where we used $d_2( x^{t+1} ,  x^{t+1,m} ) \le \kappa_1 $ (proved in (i)). Thus,
\begin{equation*}
\| Wx^{t+1} \|_{\infty} \le  (C_1'' + 2C_2' \kappa_1) \sqrt{n \log n} = \kappa_2 \sqrt{n \log n}.
\end{equation*}

(iii) \underline{Verifying (\ref{ineqn::indct3}) for $t+1$}: since $\mathcal{P} z = z$, by Lemma \ref{lem::P0}, we deduce
\begin{align}\label{ineqn::iii}
d_2 ( x^{t+1} , z ) = d_2( \mathcal{P} \mathcal{L} x^t, \mathcal{P} z ) \le (1-\varepsilon_0)^{-1} d_2( \mathcal{L} x^t , z ),
\end{align}
where $\varepsilon_0 = 1 - \min_k | (\mathcal{L} x^t)_k |$. This is only informative if $\min_k | (\mathcal{L} x^t)_k | > 0$. Observe
\begin{align*}
\mtheta  \| \etheta  \mathcal{L} x^t - z \|_{\infty} &\le \mtheta \| (\etheta   (z^* x^t )/n  - 1) z \|_{\infty} + \sigma \| W x^t \|_{\infty}/n, \\
&\le  \mtheta |  \etheta  (z^* x^t )/n  - 1 | + \kappa_2 \sigma \sqrt{\log n/n},
\end{align*}
where we used $\| W x^t \|_{\infty} \le \kappa_2 \sqrt{n\log n}$ (by the induction hypothesis). Since
\begin{equation*}
 \mtheta |  \etheta  (z^* x^t )/n  - 1 | = 1 - |z^* x^t | /n = \frac{d_2(x^t, z)^2}{2n} \le \kappa_3^2/2,
\end{equation*}
and $\kappa_2 \sigma \sqrt{\log n/n} \le 1/240 = \kappa_3/4$, and also $\kappa_3 =1/60$, it follows that 
\begin{equation*}
d_\infty ( \mathcal{L} x^t, z ) = \mtheta  \| \etheta  \mathcal{L} x^t - z \|_{\infty} \le \kappa_3/120 + \kappa_3/4 < \kappa_3/2.
\end{equation*}
This a fortiori verifies $\min_k | (\mathcal{L} x^t)_k | \ge 1 - \kappa_3/2 > 1/2$, and thus from (\ref{ineqn::iii}), 
\begin{equation}\label{ineqn::x2norm}
d_2( x^{t+1} , z )  \le 2 d_2( \mathcal{L} x^t , z ) \le 2\sqrt{n} \, d_\infty (  \mathcal{L} x^t , z ) < \kappa_3 \sqrt{n} .
\end{equation}
Here, the inequality $d_2(w,v) \le \sqrt{n}\, d_{\infty}(w,v)$ ($\forall w,v \in \mathbb{C}_1^n$) is obtained by considering $d_2(w, v) \leq \| e^{i \theta} w - v \|_2 \le \sqrt{n}  \| e^{i \theta} w - v \|_{\infty} = d_\infty(w, v)$ for the choice of $\theta$ determined by $d_\infty$. 

\underline{Finally}, we use induction and deduce that, for any $ t = 0,1\ldots, T$, the desired bounds (\ref{ineqn::indct1})--(\ref{ineqn::indct3}) hold with probability $1 - \mathcal{O}(n^{-2})$.
\end{proof}

\begin{proof}[Proof of Theorem \ref{thm::geoDcr}]
In Theorem \ref{thm::indct}, we showed that with probability $1 - \mathcal{O}(n^{-2})$, $x^t \in \mathcal{N}$ for $t \le T$. Now we can use Lemma \ref{lem::ctrT}, which yields
\begin{equation*}
d_2 (\mathcal{T} x^t , \mathcal{T} x^{t-1} ) \le \rho \cdot d_2( x^t , x^{t-1} ), \qquad \forall \, 1 \le t \le T-1,
\end{equation*}
where $\rho = (1-\varepsilon)^{-1}(6\kappa_3 + \sigma \|W\|_2 / n)$, and $\varepsilon = \kappa_3^2/2 + \kappa_2 \sigma \sqrt{\log n /n} < 1/2$. Thus, $\rho \le 2 \cdot (0.1 + \sigma C_2' /\sqrt{n}) \le 2 \cdot (0.1 + 1/120\sqrt{2})  < 1/2$. This leads to
\begin{equation*}
d_2( x^{t+1} ,  x^{t} ) \le \frac{1}{2} \, d_2( x^t , x^{t-1} ), \qquad \forall 1 \le t \le T-1,
\end{equation*}
and thus $d_2 ( x^{T} ,  x^{T-1} ) \le 2^{1 - T} d_2(x^1,x^0)  \le 2^{1 - T} (\|x^1\|_2 + \|x^0 \|_2) = 2^{2-T} \sqrt{n}$.
\end{proof}

\begin{proof}[Proof of Theorem \ref{thm::finCnvg}]
Consider the first part of the theorem. Let us use induction on $k$ to establish~(\ref{ineqn::y}). Trivially, (\ref{ineqn::y}) holds for $k = 0$. Now suppose for all $\ell \le k$, $d_2( x^{T+\ell} , x^{T+\ell-1} ) \le 2^{-\ell} d_2( x^{T} , x^{T-1})$. We now show the same for $k+1$. First,
\begin{equation*}
d_2 (x^{T+k} , x^{T-1} ) \le \sum_{\ell=0}^k d_2( x^{T+\ell} , x^{T+\ell-1} ) \le  \frac{\kappa_3}{4} \sum_{\ell=0}^k2^{-\ell} < \frac{\kappa_3}{2}. 
\end{equation*}
Similarly, $ d_2( x^{T+k-1} , x^{T-1} ) < \kappa_3/2$. By assumption, $ d_2( x^{T-1} , z ) \le \kappa_3 \sqrt{n}$, so 
\begin{equation}\label{ineqn::xT+k}
d_2( x^{T+k} , z ) \le 3\kappa_3\sqrt{n}/2, \qquad d_2 ( x^{T+k-1} , z ) \le 3\kappa_3\sqrt{n}/2.
\end{equation}
Moreover, 
\begin{align}
\| W x^{T + k} \|_{\infty} &\le \| W x^{T-1} \|_{\infty} + \mtheta \| W ( \etheta x^{T + k} - x^{T-1} ) \|_2 \notag \\
&\le \kappa_2\sqrt{n \log n} + \| W \|_2  \cdot d_2( x^{T + k} , x^{T-1} )  \notag \\
&\le (\kappa_2 + C_2'\kappa_3) \sqrt{n \log n}, \label{ineqn::WxT+k}
\end{align}
and similarly $\| W x^{T + k-1} \|_{\infty} \le  (\kappa_2 + C_2'\kappa_3) \sqrt{n \log n}$. Using Lemma \ref{lem::ctrT}, we have
\begin{align}
d_2( \mathcal{T} x^{T+k} , \mathcal{T} x^{T+k-1} ) &\le (1 - \varepsilon)^{-1}(9\kappa_3 + \sigma \|W\|_2 / n) \cdot d_2( x^{T+k} , x^{T+k-1} ) \notag\\
&\le 2\left(\frac{9}{60} + \frac{1}{120\sqrt{2}} \right) \cdot d_2( x^{T+k} , x^{T+k-1} ) \notag \\
& \le \frac{1}{2} d_2( x^{T+k} , x^{T+k-1} ), \label{ineqn::similar}
\end{align}
where $\varepsilon \le (\frac{3}{2}\kappa_3)^2/2 + (\kappa_2 + C_2'\kappa_3) \sigma \sqrt{\log n / n } < 1/2.$ By the induction hypothesis, $d_2( x^{T+k} , x^{T+k-1} ) \le 2^{-k} d_2( x^T , x^{T-1} )$, so $d_2( \mathcal{T} x^{T+k} , \mathcal{T} x^{T+k-1} ) \le 2^{-k-1} d_2( x^T , x^{T-1} )$. This completes the induction and proves (\ref{ineqn::y}).

Now consider the second part of the theorem. The existence of the limit follows if we can show the metric space $\mathbb{C}_1^n /\! \sim$ is complete. For any Cauchy sequence $\{[u^t]\}_{t=1}^\infty$ in $\mathbb{C}_1^n / \! \sim$, we can choose a subsequence $\{[u^{t_k}]\}_{k=1}^\infty$ such that $d_2([u^{t_k}] , [u^{t_{k+1}}]) \le 2^{-k}$. Now we define a sequence $\{y^k\}_{k=1}^\infty$ in $\mathbb{C}_1^n$ by fixing the phases of $[u^{t_k}]$ iteratively. Given an arbitrary representative $u^{t_1}$ in the equivalence class $[u^{t_1}]$, we set $y^1 = u^{t_1}$. Given $y^k$, we can choose a suitable $\theta \in \mathbb{R}$ such that $\| \etheta u^{t_{k+1}} - y^k \|_2$ attains $d_2(u^{t_{k+1}}, y^k)$. Then we set $y^{k+1} = \etheta u^{t_{k+1}}$. This produces a Cauchy sequence $\{y^k\}_{k=1}^{\infty}$ in $\mathbb{C}_1^n$, so it admits a limit $y^{\infty} = \lim_k y^k \in \mathbb{C}_1^n$. It follows, thus, that $[y^\infty]$ is the limit of $\{[u^{t_k}]\}_{k=1}^\infty$ in $\mathbb{C}_1^n / \! \sim$. Moreover, it is easy to see that $[y^\infty]$ is also the limit of $\{[u^t]\}_{t=1}^\infty$. 

Hence, $[x^t]$ converges to a limit $[x^\infty]$ in $\mathbb{C}_1^n / \! \sim$. The inequalities in (\ref{ineqn::Wyinf}) follows from (\ref{ineqn::xT+k}), (\ref{ineqn::WxT+k}) and the continuity argument.



Finally, we will show that $x^\infty$ is a fixed point of $\mathcal{T}$. Because all iterates (including $x^\infty$) satisfy (\ref{ineqn::xT+k})  and \eqref{ineqn::WxT+k}, we can use Lemma \ref{lem::ctrT} and derive 
$d_2 ( \mathcal{T} x^{\infty} , \mathcal{T}x^{T+k-1} ) \le d_2( x^{\infty} , x^{T+k-1} )  / 2$ (similarly as in (\ref{ineqn::similar})). The right-hand side is vanishing as $k\to\infty$, so
\begin{equation*}
d_2( \mathcal{T} x^{\infty} , x^{\infty} ) \le \lim_{k \to \infty} \big( d_2( \mathcal{T} x^{\infty} , x^{T+k} ) + d_2(  x^{T+k} , x^{\infty}) \big) = 0.
\end{equation*}
This implies that there exists some $\theta \in \mathbb{R}$ such that $ \etheta \mathcal{T} x^{\infty} = x^{\infty}$. Since $\min_k | (\mathcal{L}x^{\infty})_k | > 1 - \varepsilon > 0$ (as in~\eqref{ineqn::similar}), we can rewrite it as
\begin{equation*}
\etheta \mathcal{L} x^{\infty}  =n^{-1} \diag(\mu) x^{\infty},
\end{equation*}
where $\mu_k = n| (\mathcal{L} x^{\infty})_k |$. The above identity implies $e^{i \theta} (x^\infty)^* \mathcal{L} x^\infty = \sum_{k=1}^n \mu_k / n > 0$. Since $\mathcal{L}$ is Hermitian, the left-hand side must be real, so $e^{i \theta} \in \{ \pm 1\}$. Note that $e^{i \theta}$ must be $1$, since
\begin{align*}
(x^\infty)^* \mathcal{L} x^\infty &= \frac{1}{n} \left| (x^\infty)^* z \right|^2 + \frac{1}{n} \sigma  x^* W x \ge \frac{1}{n}(n - d_2(x^\infty, z)^2/2)^2 - \sigma  \| W \|_2 \\
& \ge n ( 1 - (3\kappa_3/2)^2/2 )^2 - C_2'\sigma  \sqrt{n} = n(1 - \frac{1}{3200})^2 - C_2'\sigma  \sqrt{n},
\end{align*}
which is positive under condition (\ref{ineqn::sigmaBound}). Replacing $\mathcal{L}$ with $C/n$, we finish the proof.
\end{proof}

\subsection*{Proofs for Section~\ref{sec:mainres}}
This subsection presents the proofs of the main results Theorems~\ref{thm:mainGPM2}--\ref{thm:mainbound2}. Note Theorem \ref{thm::mainEig2} is already proved in Section \ref{sec:eigen}, and Theorem \ref{thm:mainSDP2} is proved in Section~\ref{sec:phaseSyn}.

\begin{proof}[Proof of Theorem \ref{thm:mainGPM2}]
With probability $1 - \mathcal{O}(n^{-2})$, (\ref{ineqn::cauchy}) and (\ref{ineqn::y}) hold. With the trivial inequality $d_2(x^1,x^0) \le 2\sqrt{n}$, we have $d_2(x^t, x^{t-1}) \le 2^{2-t} \sqrt{n}$ for all $t$. By the triangular inequality,
\begin{equation*}
d_2(x^t, x^\infty) \le \sum_{k=t}^\infty d_2( x^{k+1}, x^k) \le \sum_{k=t}^\infty 2^{1-k} \sqrt{n} = 2^{2-t} \sqrt{n}.
\end{equation*}
This proves Theorem \ref{thm:mainGPM2}.
\end{proof}
\begin{proof}[Proof of Theorem \ref{thm:mainbound2}]
Without loss of generality, suppose $n \ge 2$ and fix the global phase of $x^\infty$ such that $z^* x^\infty = | z^* x^\infty |$. Also suppose the constant $c_0$ satisfies (\ref{ineqn::sigmaBound}). 

(i) We first show $\| x^\infty - z \|_2 \le 4C_2'\sigma$. This follows from the same argument in \citep[Lemma 4.1]{bandeira2014tightness}. Note that due to the sub-gaussian assumption in this paper, a difference is the bound on $\| W \|_2$: $\| W \|_2 \le C_2' \sqrt{n}$ by Lemma \ref{lem::conctr2} (whereas $\| W \|_2 \le 3 \sqrt{n}$ in \cite{bandeira2014tightness}).

(ii) In the last subsection ``Verifying optimality'' of Section \ref{sec:phaseSyn}, we showed, using the conclusions of Theorem \ref{thm::finCnvg}, that $x^\infty$ is the unique optimum of (\ref{eq:P}) up to phase with probability $1 - \mathcal{O}(n^{-2})$.

(iii) Next we show $\| x^\infty - z \|_{\infty} \le C \sigma \sqrt{\log n/n}$ where $C>0$ is some constant. From (i), $| z^* x^\infty | = n - \| x^\infty - z \|_2^2/2 \ge n - 8(C_2')^2 \sigma^2$. In the proof of Lemma 4.2 in \citep{bandeira2014tightness}, it is shown that
\begin{align*}
|z^*x^\infty| \| x^\infty-  z \|_\infty \le 2\sigma \| W x^\infty \|_\infty.
\end{align*}
Therefore, it follows that
\begin{align*}
\| x^\infty-  z \|_\infty &\le 2n^{-1}\sigma \| W x^\infty \|_\infty + 8(C_2')^2 n^{-1} \sigma^2 \| x^\infty - z \|_\infty \\
&\le 2n^{-1}\sigma (\kappa_2 + C_2'\kappa_3) \sqrt{n \log n} + 16(C_2')^2 n^{-1} \sigma^2,
\end{align*}
where we used (\ref{ineqn::Wyinf}) in Theorem \ref{thm::finCnvg} and a trivial bound $\| x^\infty - z \|_\infty \le 2$. Since $\sigma = \mathcal{O}(\sqrt{n / \log n})$, we conclude that there exists a constant $C$ such that $\| x^\infty - z \|_{\infty} \le C \sigma \sqrt{\log n/n}$.
\end{proof}

\footnotesize{
\bibliographystyle{plain}
\bibliography{boumal}

\begin{thebibliography}{10}

\bibitem{bandeira2015laplacian}
A.S. Bandeira.
\newblock Random {L}aplacian matrices and convex relaxations.
\newblock {\em arXiv preprint arXiv:1504.03987}, 2015.

\bibitem{bandeira2014tightness}
A.S. Bandeira, N.~Boumal, and A.~Singer.
\newblock Tightness of the maximum likelihood semidefinite relaxation for
  angular synchronization.
\newblock {\em Mathematical Programming}, pages 1--23, 2016.

\bibitem{bandeira2014open}
A.S. Bandeira, Y.~Khoo, and A.~Singer.
\newblock Open problem: Tightness of maximum likelihood semidefinite
  relaxations.
\newblock In Csaba~Szepesvári Maria Florina~Balcan, Vitaly~Feldman, editor,
  {\em Proceedings of the 27th Conference on Learning Theory}, volume~35 of
  {\em JMLR W\&CP}, pages 1265--1267, 2014.

\bibitem{bean2013optimal}
Derek Bean, Peter~J Bickel, Noureddine El~Karoui, and Bin Yu.
\newblock Optimal m-estimation in high-dimensional regression.
\newblock {\em Proceedings of the National Academy of Sciences},
  110(36):14563--14568, 2013.

\bibitem{Ben11}
Florent Benaych-Georges and Raj~Rao Nadakuditi.
\newblock The eigenvalues and eigenvectors of finite, low rank perturbations of
  large random matrices.
\newblock {\em Advances in Mathematics}, 227(1):494--521, 2011.

\bibitem{crbsubquot}
N.~Boumal.
\newblock On intrinsic {Cram\'er-Rao} bounds for {Riemannian} submanifolds and
  quotient manifolds.
\newblock {\em Signal Processing, IEEE Transactions on}, 61(7):1809--1821,
  2013.

\bibitem{boumal2015staircase}
N.~Boumal.
\newblock A {R}iemannian low-rank method for optimization over semidefinite
  matrices with block-diagonal constraints.
\newblock {\em arXiv preprint arXiv:1506.00575}, 2015.

\bibitem{boumal2016nonconvexphase}
N.~Boumal.
\newblock Nonconvex phase synchronization.
\newblock {\em SIAM Journal on Optimization}, 26(4):2355--2377, 2016.

\bibitem{boumal2013MLE}
N.~Boumal, A.~Singer, and P.-A. Absil.
\newblock Robust estimation of rotations from relative measurements by maximum
  likelihood.
\newblock In {\em {Decision and Control (CDC), 2013 IEEE 52nd Annual Conference
  on}}, pages 1156--1161, Dec 2013.

\bibitem{capitaine2009largest}
M.~Capitaine, C.~Donati-Martin, and D.~F{\'e}ral.
\newblock The largest eigenvalues of finite rank deformation of large {W}igner
  matrices: convergence and nonuniversality of the fluctuations.
\newblock {\em The Annals of Probability}, 37(1):1--47, 2009.

\bibitem{chen2016projected}
Yuxin Chen and Emmanuel Candes.
\newblock The projected power method: An efficient algorithm for joint
  alignment from pairwise differences.
\newblock {\em arXiv preprint arXiv:1609.05820}, 2016.

\bibitem{cucuringu2011eigenvector}
M.~Cucuringu, A.~Singer, and D.~Cowburn.
\newblock Eigenvector synchronization, graph rigidity and the molecule problem.
\newblock {\em Information and Inference, a Journal of the IMA}, 1(1):21--67,
  2012.

\bibitem{cucuringu2012sensor}
Mihai Cucuringu, Yaron Lipman, and Amit Singer.
\newblock Sensor network localization by eigenvector synchronization over the
  euclidean group.
\newblock {\em ACM Transactions on Sensor Networks (TOSN)}, 8(3):19, 2012.

\bibitem{DavKah70}
Chandler Davis and William~Morton Kahan.
\newblock The rotation of eigenvectors by a perturbation. iii.
\newblock {\em SIAM Journal on Numerical Analysis}, 7(1):1--46, 1970.

\bibitem{deshpande2014conePCA}
Y.~Deshpande, A.~Montanari, and E.~Richard.
\newblock Cone-constrained principal component analysis.
\newblock In Z.~Ghahramani, M.~Welling, C.~Cortes, N.D. Lawrence, and K.Q.
  Weinberger, editors, {\em Advances in Neural Information Processing Systems
  27}, pages 2717--2725. Curran Associates, Inc., 2014.

\bibitem{FanWanZho16}
Jianqing Fan, Weichen Wang, and Yiqiao Zhong.
\newblock An $\ell^{\infty}$ eigenvector perturbation bound and its application
  to robust covariance estimation.
\newblock {\em arXiv preprint arXiv:1603.03516}, 2016.

\bibitem{giridhar2006distributed}
A.~Giridhar and P.R. Kumar.
\newblock Distributed clock synchronization over wireless networks: Algorithms
  and analysis.
\newblock In {\em Decision and Control, 2006 45th IEEE Conference on}, pages
  4915--4920. IEEE, 2006.

\bibitem{holland1977robust}
Paul~W Holland and Roy~E Welsch.
\newblock Robust regression using iteratively reweighted least-squares.
\newblock {\em Communications in Statistics-theory and Methods}, 6(9):813--827,
  1977.

\bibitem{javanmard2016phase}
Adel Javanmard, Andrea Montanari, and Federico Ricci-Tersenghi.
\newblock Phase transitions in semidefinite relaxations.
\newblock {\em Proceedings of the National Academy of Sciences},
  113(16):E2218--E2223, 2016.

\bibitem{JohLu12}
Iain~M Johnstone and Arthur~Yu Lu.
\newblock On consistency and sparsity for principal components analysis in high
  dimensions.
\newblock {\em Journal of the American Statistical Association}, 2012.

\bibitem{journee2010generalized}
M.~Journ{\'e}e, Y.~Nesterov, P.~Richt{\'a}rik, and R.~Sepulchre.
\newblock Generalized power method for sparse principal component analysis.
\newblock {\em The Journal of Machine Learning Research}, 11:517--553, 2010.

\bibitem{lelarge2016fundamental}
Marc Lelarge and L{\'e}o Miolane.
\newblock Fundamental limits of symmetric low-rank matrix estimation.
\newblock {\em arXiv preprint arXiv:1611.03888}, 2016.

\bibitem{liu2016statistical}
H.~Liu, M.-C. Yue, and A.~M.-C. So.
\newblock On the estimation performance and convergence rate of the generalized
  power method for phase synchronization.
\newblock {\em arXiv preprint arXiv:1603.00211}, 2016.

\bibitem{luss2013conditional}
R.~Luss and M.~Teboulle.
\newblock Conditional gradient algorithms for rank-one matrix approximations
  with a sparsity constraint.
\newblock {\em SIAM Review}, 55(1):65--98, 2013.

\bibitem{martinec2007robust}
Daniel Martinec and Tomas Pajdla.
\newblock Robust rotation and translation estimation in multiview
  reconstruction.
\newblock In {\em Computer Vision and Pattern Recognition, 2007. CVPR'07. IEEE
  Conference on}, pages 1--8. IEEE, 2007.

\bibitem{ozyesil2016synchronization}
Onur Ozyesil, Nir Sharon, and Amit Singer.
\newblock Synchronization over {Cartan} motion groups via contraction.
\newblock {\em arXiv preprint arXiv:1612.00059}, 2016.

\bibitem{Pau07}
Debashis Paul.
\newblock Asymptotics of sample eigenstructure for a large dimensional spiked
  covariance model.
\newblock {\em Statistica Sinica}, pages 1617--1642, 2007.

\bibitem{perry2016optimality}
A.~Perry, A.S. Wein, A.S. Bandeira, and A.~Moitra.
\newblock Optimality and sub-optimality of {PCA} for spiked random matrices and
  synchronization.
\newblock {\em arXiv preprint arXiv:1609.05573}, 2016.

\bibitem{RohChaYu11}
Karl Rohe, Sourav Chatterjee, and Bin Yu.
\newblock Spectral clustering and the high-dimensional stochastic blockmodel.
\newblock {\em The Annals of Statistics}, pages 1878--1915, 2011.

\bibitem{rosen2016certifiably}
David~M Rosen, Luca Carlone, Afonso~S Bandeira, and John~J Leonard.
\newblock A certifiably correct algorithm for synchronization over the special
  euclidean group.
\newblock {\em arXiv preprint arXiv:1611.00128}, 2016.

\bibitem{roulet2015renegar}
V.~Roulet, N.~Boumal, and A.~d'Aspremont.
\newblock {R}enegar's condition number and compressed sensing performance.
\newblock {\em arXiv preprint arXiv:1506.03295}, 2015.

\bibitem{RudVer10}
Mark Rudelson and Roman Vershynin.
\newblock Non-asymptotic theory of random matrices: extreme singular values.
\newblock {\em arXiv preprint arXiv:1003.2990}, 2010.

\bibitem{shkolnisky2012viewing}
Yoel Shkolnisky and Amit Singer.
\newblock Viewing direction estimation in cryo-em using synchronization.
\newblock {\em SIAM journal on imaging sciences}, 5(3):1088--1110, 2012.

\bibitem{singer2010angular}
A.~Singer.
\newblock Angular synchronization by eigenvectors and semidefinite programming.
\newblock {\em Applied and Computational Harmonic Analysis}, 30(1):20--36,
  2011.

\bibitem{stella2009angular}
X~Yu Stella.
\newblock Angular embedding: from jarring intensity differences to perceived
  luminance.
\newblock In {\em Computer Vision and Pattern Recognition, 2009. CVPR 2009.
  IEEE Conference on}, pages 2302--2309. IEEE, 2009.

\bibitem{Ver10}
Roman Vershynin.
\newblock Introduction to the non-asymptotic analysis of random matrices.
\newblock {\em arXiv preprint arXiv:1011.3027}, 2010.

\bibitem{Von07}
Ulrike Von~Luxburg.
\newblock A tutorial on spectral clustering.
\newblock {\em Statistics and computing}, 17(4):395--416, 2007.

\bibitem{wang2012LUD}
L.~Wang and A.~Singer.
\newblock Exact and stable recovery of rotations for robust synchronization.
\newblock {\em Information and Inference}, 2(2):145--193, 2013.

\bibitem{zhang2006complex}
S.~Zhang and Y.~Huang.
\newblock Complex quadratic optimization and semidefinite programming.
\newblock {\em SIAM Journal on Optimization}, 16(3):871--890, 2006.

\end{thebibliography}
}

\end{document}